\documentclass[reqno, a4paper]{amsart}
\usepackage{amssymb, amscd, amsfonts,  amsthm}

\usepackage[mathscr]{eucal}
 
 \usepackage{tikz}
\usepackage{a4wide}

\newtheorem{theorem}{Theorem}[section]
\newtheorem{lemma}[theorem]{Lemma}
\newtheorem{corollary}[theorem]{Corollary}
\newtheorem{proposition}[theorem]{Proposition}

\theoremstyle{definition}
\newtheorem{definition}[theorem]{Definition}
\newtheorem{remark}[theorem]{Remark}

\newtheorem{example}[theorem]{Example}

\newcommand{\restrict}{\,{\mathbin{\vert\mkern-0.3mu\grave{}}}\,}

\newcommand{\luk}{\L u\-ka\-s\-ie\-wicz}

\newcommand{\remove}[1]{}

\newcommand{\MV}{\mathsf{MV}}
\newcommand{\IMV}{\mathsf{IMV}}
\newcommand{\class}[1]{\mathsf{#1}}
\newcommand{\F}{\mathbf{F}}

\DeclareMathOperator{\McN}{\mathcal M}

\DeclareMathOperator{\I}{[0,1]}

\DeclareMathOperator{\iii}{\rm i}
\DeclareMathOperator{\too}{\Rightarrow}

 \title[Interval MV-algebras and generalizations]
{Interval  
 MV-algebras and
generalizations}

\author{\sc Leonardo Manuel Cabrer and Daniele Mundici}
\thanks{This research was supported by a 
 Marie Curie Intra European Fellowship 
 within the 7th European Community Framework 
 Program (ref. 299401-FP7-PEOPLE-2011-IEF)}

\address[L.M. Cabrer]{Department of Statistics,
Computer Science and Applications,  ``Giu\-sep\-pe Parenti''\\ 
University of Florence\\
Viale Morgagni 59 
50134\\ Florence \\
Italy}
\email{ l.cabrer@disia.unifi.it }

\address[D. Mundici]{Department of
Mathematics and Computer Science  ``Ulisse Dini'' \\
University of Florence\\
Viale Morgagni 67/A \\
I-50134 Florence \\
Italy}
\email{ mundici@math.unifi.it }

\date{\today}

\begin{document}


   \begin{abstract}
For any MV-algebra $A$ we equip the set $I(A)$ of intervals
in  $A$ with
pointwise \luk\ negation $\neg x=\{\neg \alpha\mid \alpha\in x\}$,
 (truncated)  Minkowski sum, 
$x\oplus y=\{\alpha\oplus \beta\mid \alpha \in x,\,\,\beta\in y\}$,
pointwise \luk\ conjunction
$x\odot y=\neg(\neg x\oplus \neg y)$,
  the  operators $\Delta x=[\min x,\min x]$, 
$\nabla x=[\max x,\max x]$, and distinguished constants
$0=[0,0],\,\, 1=[1,1],\,\,\, \iii = A.$ We list a few equations
satisfied by the  algebra
$\mathcal I(A)=(I(A),0,1,\iii,\neg,\Delta,\nabla,\oplus,\odot)$,
 call {\it IMV-algebra} every model of
these equations,
 and show that, conversely,  every IMV-algebra
  is isomorphic to the IMV-algebra $\mathcal I(B)$
of all intervals in  some MV-algebra $B$. 
We show that IMV-algebras are categorically
equivalent to MV-algebras, and give a representation of
free IMV-algebras. 
We construct   {\it \luk\  interval logic}, with its
coNP-complete consequence relation,
which we prove to be complete 
for $\mathcal I(\I)$-valuations. 
For any class $\class{Q}$ of partially
ordered algebras with operations that are
monotone or antimonotone in each varlable,
we  consider the 
generalization   $\mathcal I_{\class{Q}}$  of 
the  MV-algebraic functor $\mathcal I$, and give
 necessary and sufficient conditions
for  $\mathcal I_{\class{Q}}$  to be a categorical equivalence.
These conditions are  satisfied, e.g., by all subquasivarieties of residuated lattices.  
 \end{abstract}

 \maketitle
 
 \bigskip
 
  \bigskip
 
\section{Foreword}
As shown in \cite[\S 1.6]{mun11}, 
truth-values  in 
   \luk\ logic may be thought of as arising
   from normalized  measurements
   of bounded physical observables, just as
   boolean truth-values arise from $\{\mbox{yes,no}\}$-observables. 
   \luk\ implication   is uniquely characterized
   among all  binary operations on $\I$ 
   by the Smets-Magrez theorem,   \cite{bac, smemag},
   as the only $\I$-valued continuous
   map on  $\I^2$ satisfying
  natural monotonicity conditions with respect
   to the natural order of $\I$. These conditions yield the classical
   \luk\ axioms of infinite-valued logic \L$_\infty$
   which, via Modus Ponens,
   determine the consequence relation of   \L$_\infty$.
   Closing a circle of ideas about truth-values as
   real numbers, we recover the intended meaning
   of  \L$_\infty$-formulas, by a well known completeness theorem
   \cite{rosros}, \cite{cha2}, \cite[\S 2.5]{cigdotmun}
  to the effect that   \L$_\infty$-tautologies (i.e.,   \L$_\infty$-consequences
   of the \luk\ axioms) coincide with  formulas
taking value 1 for every  $\I$-valuation.

To achieve greater adherence to
actual physical measurements---and more generally, to formally handle
the imprecise estimations/evaluations of everyday life---one
 might envisage logics whose truth-values are
the closed intervals in  $\I$.
This paper is devoted to developing the  algebraic and categorical tools 
for the construction of such logics, using \luk\ logic as a template. 
Mimicking the approach to \luk\ logic via MV-algebras,
 the set $I(A)$ of intervals
 in any MV-algebra $A$ is equipped with
pointwise \luk\ negation $\neg x=\{\neg \alpha\mid \alpha\in x\}$,
 (truncated) {\it Minkowski sum}, 
$x\oplus y=\{\alpha\oplus \beta\mid \alpha \in x,\beta\in y\}$,
pointwise \luk\ conjunction
$x\odot y=\neg(\neg x\oplus \neg y)$,
  the  operators $\Delta x=[\min x,\min x]$, 
$\nabla x=[\max x,\max x]$, and distinguished constants
$0=[0,0],\,\, 1=[1,1],\,\,\, \iii = A.$

We list nineteen simple equations   \eqref{equation:associative}-\eqref{equation:great}
satisfied by the  algebra
\[\mathcal I(A)=(I(A),0,1,\iii,\neg,\Delta,\nabla,\oplus,\odot),\]
 and call {\it IMV-algebra} every model of
these equations.
The adequacy of these equations is shown
in Theorem \ref{theorem:representation}
stating that, conversely,  every IMV-algebra
  is isomorphic to the IMV-algebra $\mathcal I(B)$
of all intervals in  some MV-algebra $B$. 
While no IMV-algebra reduct is an MV-algebra,
the functor $\mathcal I$ establishes a categorical equivalence 
between MV-algebras and
  IMV-algebras  (Theorem \ref{theorem:equivalence}).

In  Section \ref{section:completeness}
 we   obtain the  following equational completeness theorem
(Theorem \ref{theorem:hsp}):
an equation
is satisfied by all IMV-algebras iff it is satisfied by the
IMV algebra  $\mathcal I(\I)$ of all intervals in  the standard
MV-algebra $\I$  iff it is derivable from
the IMV-axioms  \eqref{equation:associative}-\eqref{equation:great}
by the familiar  rules of replacing equals by equals according to
Birkhoff equational logic. 
This result has a deeper algebraic counterpart in 
the representation Theorem \ref{theorem:free} 
of free  IMV-algebras.

Every IMV-algebra $J$ is equipped  with
{\it two} types of partial orders:
the product order and the inclusion order. The first endows $J$
with a distributive lattice structure  $(J,\sqcup,\sqcap)$; the second yields an upper
semilattice structure  $(J,\cup)$.  
In the final part of
Section \ref{section:completeness} is it proved  that each operation
$\sqcup,\sqcap,\cup$ is definable from the IMV-structure.
With reference to our initial remarks on truth-values as intervals,
 IMV-algebras
can express the fundamental inclusion order  $u\subseteq v$
between actual measurements/estimations $u,v$, 
meaning that $u$ is more precise than $v$.  Within
the MV-algebraic framework, the inclusion order has no meaning,
despite MV-algebras are
categorically equivalent to IMV-algebras.

Closing another circle of ideas, about truth-values as
intervals, 
in Section \ref{section:logic} we will introduce 
{\it \luk\ interval logic}, with its consequence relation based on
the only rule of Modus Ponens, and prove
a completeness theorem for $\mathcal I(\I)$-valuations.
The  consequence problem in
this logic, just like  
  the equational theory
of IMV-algebras turns out to be coNP-complete
 (Theorem \ref{theorem:equational}, 
Corollary  \ref{corollary:conp}). 
Thus, the increased expressive power
of \luk\ interval logic with respect to
\luk\   logic does not entail greater
 complexity of the consequence problem.

 In the all-important
 Section \ref{section:general},    for any class $\class{Q}$ of partially
ordered algebras with operations that are
monotone or antimonotone in each variable,
we  consider the 
generalization   $\mathcal I_{\class{Q}}$  of 
the  MV-algebraic functor $\mathcal I$, and give
 necessary and sufficient conditions
for  $\mathcal I_{\class{Q}}$  
to be a categorical equivalence.
As shown in   Corollary \ref{corollary:extensive}, these conditions
 are satisfied, e.g.,  by every
quasivariety  $\class Q$ having a
lattice reduct---including many
 classes of  ordered algebras related with logical systems of general interest,
 such as BL-algebras, Heyting algebras,  
 G\"odel algebras,  MTL-algebras, and more generally
 every subquasivariety of residuated lattices  (Corollary \ref{corollary:listone}).

 Remarkably enough,  the pervasiveness of  the categorical equivalence 
$\mathcal I_{\class{Q}}$  has gone virtually unnoticed in the  literature on
interval and triangle algebras and their logics, interval constructors
and triangularizations.
  The final Section \ref{section:literature}   is devoted
to relating our results to the extensive  literature on this subject,
 \cite{bedbed,natal-LNCS,information,35years,cina,van,van-bis}.
 
The  only  prerequisite for this paper
is some acquaintance
 with MV-algebras \cite{cigdotmun}, Birkhoff-style universal algebra
\cite{bursan}, 
and the rudiments of category theory \cite{mcl}.

  \bigskip
\section{The equational class of IMV-algebras}\label{Sec:IMV}
Let  $A=(A,0,1,\neg,\oplus,\odot)$ be an MV-algebra.
By an {\it interval} of $A$
we mean a subset $x$ of $A$ of the form
$x=[\alpha,\beta]=\{\xi\in A\mid \alpha\leq \xi\leq \beta\}$,
where $\alpha,\beta\in A$ and $\alpha\leq \beta$.
In case  $\alpha=\beta$ we say that $x$ is
{\it degenerate}.
 We let    $I(A)$  denote
the set of intervals
in  $A$, and $D(A)\subseteq I(A)$  the set of degenerate
intervals.
We record here a first result, to the effect that
``sums of intervals are intervals'':
 
 \bigskip 
  \begin{proposition}
  \label{proposition:riesz} For any MV-algebra $A$
  and intervals  $H,K \in I(A)$, let  
  $H\oplus K$ denote the (truncated) {\em Minkowski sum} of $H$ and $K$, 
  $$H\oplus K=\{\alpha\in A\mid \alpha=\mu\oplus \nu
  \mbox{ for some $\mu\in H$ and $\nu\in K$}\}.$$
Then $H\oplus K$ is an interval in  $A$.
 Similarly,
the set $$H\odot K=\{\alpha\in A\mid \alpha=\mu\odot \nu
  \mbox{ for some $\mu\in H$ and $\nu\in K$}\}$$
   is an interval in  $A$.
 \end{proposition}

 \begin{proof}  We first prove the following special case:
 \begin{equation}
\label{equation:riesz}
[0,\alpha]\oplus[0, \beta]=[0,\alpha\oplus \beta]
\,\,\mbox{  for all \,\,$\alpha,\beta\in A$}.
\end{equation}
We will make use of the derived lattice operations  $\wedge,\vee$ 
and the natural order $\leq$ of $A$, \cite[\S 1.1]{cigdotmun}.
Given  $x\in [0,\alpha\oplus \beta]$  let
$\beta'=\beta\wedge x$  and $\alpha'=x\odot\neg \beta'.$
Then
$$
\alpha'\oplus\beta'=(x\odot\neg\beta')\oplus\beta'=\neg(\neg x\oplus\beta')\oplus\beta'
= x\vee \beta'=x.
$$
Further,
$
\alpha'=x\odot\neg\beta'=x\odot\neg(x\wedge \beta)=x\odot(\neg x\vee \neg \beta)
= (x\odot\neg x) \vee (x\odot\neg\beta)=x\odot\neg\beta.
$
 Now from
 $
 \alpha'\odot\neg\alpha=(x\odot\neg\beta)\odot\neg\alpha
 =x\odot\neg(\beta\oplus\alpha)\leq x\odot\neg x=0
 $
 we get $\alpha'\leq \alpha.$ Since  $\beta'\leq\beta,$
 \eqref{equation:riesz} is settled.
 One now easily proves $[\alpha,\alpha]\oplus[0, \beta]
 =[\alpha,\alpha\oplus \beta]$, and more generally,
 $[\alpha,\alpha]\oplus[\delta, \beta]
 =[\alpha\oplus\delta,\alpha\oplus \beta]$.
A final verification using all these preliminary results
yields the desired conclusion $[\delta,\alpha]\oplus[\theta,\beta]
=[\delta\oplus\theta,\alpha\oplus\beta]$.
Since in every MV-algebra  $\alpha\odot\beta=\neg(\neg \alpha \oplus
\neg\beta)$, one immediately verifies that $H\odot K$ is an interval in  $A$.
(Readers familiar with the categorical equivalence
$\Gamma$  between MV-algebras and unital
abelian $\ell$-groups  (\cite[\S 3]{mun86}) will observe that 
\eqref{equation:riesz} is a special case of the
Riesz decomposition property
(\cite[Lemma 1, page 310, and Theorem 49, p. 328]{bir})
 of the unital
abelian  $\ell$-group   $(G,u)$ defined by
$\Gamma(G,u)=A.$)
 \end{proof}

In view of the foregoing result, for any MV-algebra $A$,
the set $I(A) $ is
made  into an algebra
$\mathcal I(A)=(I(A),
 0,1,\iii, \neg,\Delta,\nabla, \oplus,\odot )$ of type
  $(0,0,0,1,1,1,2,2)$, 
by equipping it with the
 distinguished constants
 \begin{equation}
\label{constants}
0=[0,0],\,\,\,1=[1,1],\,\,\,\iii=A ,
\end{equation} 
the   operations
\begin{eqnarray}
\label{negation}
{\neg} x &=&\{\neg \alpha\mid \alpha \in x\},\,\,  \mbox{pointwise negation}\\
\label{oplus}
x \oplus y&=&\{\alpha\oplus \beta
  \mid \alpha \in x, \beta \in y\}, \mbox{ (truncated) Minkowski  sum}\\
  \label{odot}
  x\odot y &=&\{\alpha\odot \beta
  \mid \alpha \in x, \beta \in y\}, \mbox{ pointwise \luk\ conjunction},
  \end{eqnarray}
  and the lower/upper collapse operations
    \begin{eqnarray} 
  \label{delta}
  \Delta x &=&[\min x,\min x], \\
\label{nabla}
\nabla x &=& [\max x,\max x],
\end{eqnarray}
where the min (resp., the max) of an interval
$x=[\alpha,\beta]\in I(A)$  equals
 $\alpha$ 
(resp., equals $\beta$), according to the
 natural order $\leq$ of $A$.
 From the context it will always be clear whether the constants
 $0,1$ and the operations
 $\neg,\oplus,\odot$ are those of $\mathcal I(A)$ or those of $A$.

Form the definition of $\mathcal{I}(A)$, the set
$D(A)\subseteq I(A)$ is the universe of a subalgebra  $\mathcal{D}(A)$  of the MV-algebra reduct of $\mathcal{I}(A)$. 
Clearly,   the map $\iota_A\colon \alpha\in A\mapsto
[\alpha,\alpha]\in D(A)$ becomes an isomorphism of $A$
onto $\mathcal D(A),$  in symbols,   
\begin{equation}
\label{equation:iota}
\iota_A\colon A\cong \mathcal D(A)\,.
\end{equation}

\bigskip

 \begin{proposition}
 \label{proposition:contemplation}
 The algebra 
 $\mathcal I(A)$
 satisfies the following equations:
\begin{eqnarray}
\label{equation:associative}
x\oplus(y\oplus z)&=&(x\oplus y)\oplus z\\
\label{equation:commutative}
x\oplus y&=&y\oplus x\\
\label{equation:opluszero}
x\oplus 0&=&x\\
\label{equation:coneutral}
x\oplus \neg 0&=&\neg 0\\
\label{equation:negneg}
\neg\neg x&=&x\\
\label{equation:mangani}
\neg(\neg\Delta x\oplus \Delta y)\oplus \Delta y&=&
\neg(\neg\Delta y\oplus \Delta x)\oplus \Delta x\\
\label{equation:odot}
x\odot y &=& \neg(\neg x\oplus \neg y)\\
\label{equation:one}
1&=&\neg 0\\
\label{equation:nabla}
\nabla x&=&\neg\Delta\neg x\\
\label{equation:negi}
\neg \iii&=&\iii\\
\label{equation:deltazero}
\Delta 0&=&0\\
\label{equation:deltaone}
\Delta 1&=&1\\
\label{equation:deltai}
\Delta \iii&=&0\\
\label{equation:deltadelta}
\Delta\Delta x&=&\Delta x\\
\label{equation:deltanabla}
\Delta\nabla x&=&\nabla x\\
\label{equation:deltaoplus}
\Delta(x\oplus y)&=&\Delta x\oplus \Delta y\\
\label{equation:deltaodot}
\Delta(x\odot y)&=&\Delta x\odot \Delta y\\
\label{equation:order}
\Delta x \odot  \neg\nabla  x &=& 0\\
\label{equation:great}
\Delta x \oplus (\iii\odot \nabla x\odot  \neg\Delta x) &=&x .
\end{eqnarray}
\end{proposition}
\begin{proof}
\eqref{equation:associative}-\eqref{equation:negneg}
are easily verified properties of pointwise negation and 
Minkowski (truncated) sum \eqref{oplus}. They are inherited by
$\mathcal I(A)$ from the corresponding
properties of negation and truncated sum in $A$, \cite[Definition~1.1.1]{cigdotmun}. 

 To prove
\eqref{equation:mangani}, 
we first observe that
  both terms in this equation depend on $x$ and $y$
  only via $\Delta x$ and $\Delta y.$
Since
 $\Delta x$ and $\Delta y$ are degenerate intervals of 
 $\mathcal I(A)$, we can write
 $\Delta x =[\alpha,\alpha]=\iota_A(\alpha)$ and
  $\Delta y =[\beta,\beta]=\iota_A(\beta)$
  for some $\alpha,\beta\in A.$  Since
 $A$ is an MV-algebra, it satisfies the identity
  $\neg(\neg \alpha \oplus \beta)\oplus \beta=
  \neg(\neg \beta \oplus \alpha)\oplus \alpha$.
  Thus by  \eqref{equation:iota}, 
  $\mathcal I(A)$ satisfies
  \eqref{equation:mangani}.
 
  For the proof of \eqref{equation:odot} one
uses \eqref{odot} and the well known definability,
\cite[\S 1]{cigdotmun}, of 
$\odot$ from $\neg$ and $\oplus$ in $A$, as well
as in its isomorphic copy $\mathcal D(A)$. 
 
    Equations 
\eqref{equation:one}-\eqref{equation:deltanabla}
are immediate consequences of \eqref{constants} and
definitions
\eqref{delta}--\eqref{nabla}. 

For the verification of \eqref{equation:deltaoplus} one 
combines 
\eqref{oplus} with  the translation invariance of
the natural order of $\mathcal D(A)$, i.e.,
the distributivity of $\oplus$ and $\odot$
over $\vee$ and $\wedge$ \cite[Proposition~1.1.6]{cigdotmun}.

Equation  \eqref{equation:deltaodot} easily
follows by
\eqref{delta},  \eqref{equation:odot} and
\eqref{equation:deltaoplus}.

As a preliminary step for the proof of
  \eqref{equation:order}, one notes
that the left hand term therein  depends on  $x$ only via
 $\Delta x$ and  $\nabla x.$  Now,  
  $\Delta x, \nabla x\in D(A)$ are degenerate intervals
  in  $A$, acted upon by the  ($\iota_A$-images of) the
  MV-algebraic operations of $A$. 
  Equation
  \eqref{equation:order} states  that  for every interval $x=
  [\alpha,\beta]\in \mathcal I(A)$,
$\Delta x\leq \nabla x$,
i.e., $[\min x,\min x]\leq[\max x,\max x]$, i.e.,
$[\alpha,\alpha]\leq[\beta,\beta]$
in the natural order of $\mathcal D(A)$ inherited from
the natural order of $A$ via the isomorphism $\iota_A$. 
Going back via
  $\iota_A^{-1}$, \eqref{equation:order} 
  amounts to the inequality
  $\alpha\leq \beta$, which holds in $A$ by definition of interval.
  Thus
 $\mathcal I(A)$ satisfies \eqref{equation:order}. 

Finally, for the verification
of     \eqref{equation:great} letting $[\alpha,\beta]\in\mathcal{I}(A)$ we can write
\begin{align*}
\Delta [\alpha,\beta] \oplus ([0,1]\odot \nabla [\alpha,\beta]\odot  \neg\Delta [\alpha,\beta]) &= [\alpha,\alpha] \oplus ([0,1]\odot[\beta,\beta]\odot  \neg[\alpha,\alpha])\\
&=[\alpha,\alpha]\oplus [0,\beta\odot  \neg\alpha]=[\alpha,\alpha\oplus(\beta\odot  \neg\alpha)]\\
&=[\alpha,\beta].\qedhere
\end{align*}
\end{proof}

\bigskip 
\section{Representation  of IMV-algebras}

\begin{definition}
\label{definition:imv} 
An {\it IMV-algebra}  
 is a structure
 $J=(J,
 0,1,\iii, \neg,\Delta,\nabla, \oplus,\odot )$ 
 of type
$(0,0,0,1,1,1,2,2)$ 
 satisfying equations
\eqref{equation:associative}-\eqref{equation:great}.
The {\it center}  $\,C(J)$  of  $J$ is the set  of all
elements  $x\in J$ such that 
$
\Delta x=\nabla x.
$
\end{definition}

From Proposition \ref{proposition:contemplation}
we immediately obtain:

\begin{proposition}
\label{proposition:trivial}
For any MV-algebra $A$ let
$\mathcal I(A)
= (I(A),
 0,1,\iii, \neg,\Delta,\nabla, \oplus,\odot )$ be  the algebra
 of intervals in  $A$,  
 as defined by \eqref{constants}--\eqref{odot} in view of 
 Proposition \ref{proposition:riesz}. Then
 
\medskip
 \begin{itemize}
 \item[(i)] $\,\,\mathcal I(A)$  is  an IMV-algebra. 
 
\medskip
    \item[(ii)]
   Every IMV-algebra satisfies the following
  quasiequation:
\begin{equation}
\label{equation:quasi}
\mbox{If  }
\Delta x=\Delta y\,\,\, \mbox{and}\,\,\, \nabla x = \nabla y 
\mbox{  then  } x=y.
\end{equation}

\medskip
\item[(iii)]
 For any  IMV-algebra  $J$, its
 $(0,1,\neg,\oplus,\odot)$-reduct is not an  MV-algebra.
Indeed, the equation
 $x\oplus \neg x=1$ fails in $\mathcal I(\I)$  (already with $x=\iii$).
The  MV-axiom 
$\neg(\neg \alpha\oplus \beta)\oplus \beta
=\neg(\neg\beta\oplus\alpha)\oplus \alpha$
 fails   in $\mathcal I(\I)$  for  $x=\iii,\,\, y=1.$
 
 \medskip
 \item[(iv)]  IMV-algebras are not term-equivalent 
(\cite[\S 4]{universal}) to
 MV-algebras: indeed, 
there is no
  two-element IMV-algebra.
    \end{itemize}
\end{proposition}

\medskip
As a converse of (i) above, 
 in  Theorem \ref{theorem:representation} we
 will see that, up to isomorphism, algebras of the form
$\mathcal I(A)$ exhaust all possible IMV-algebras.

The quasiequation
\eqref{equation:quasi} in
 is reminiscent of Moisil's
 ``determination principle'', \cite[p.106]{moisil}.

\medskip

\begin{proposition}
  \label{proposition:varia}
  Let $J$ be an  IMV-algebra.

  \begin{itemize}
  \item[(i)]  For all $x\in J$,\,\,\,  $\Delta x=\nabla x
\,\,\,\mbox{\it iff}\,\,\, x=\nabla x  \,\,\,\mbox{\it iff}\,\,\,
  x=\Delta x   \,\,\,\mbox{\it iff}\,\,\, $
    $x=\Delta y$ for some $y\in J$
$\,\,\,\mbox{\it iff}\,\,\,
  x=\nabla z$ for some  $z \in J$.
  
  \medskip
  \item[(ii)]  The center $C(J)$ is closed under the operations
     $\neg,\oplus, \odot$ of $J$ and contains the
 elements $0,1$ of $J$. The resulting
     subreduct  
     $\mathcal C(J)=(C(J),0,1,\neg,\oplus,\odot)$ of $J$
     is an MV-algebra, called the\,
{\em central}  MV-algebra of  $J.$

\medskip
\item[(iii)]     For any IMV-algebra $K$  
     and homomorphism  $\theta\colon J\to K$, the
     restriction $\theta\restrict C(J)$ of $\theta$ to 
      $C(J)$   is a homomorphism
     of $\mathcal C(J)$ into  $\mathcal C(K).$
 
 \medskip    
\item[(iv)] 

Let us define the
 map  $\gamma_J\colon J\to \mathcal{I}(\mathcal{C}(J))$ by the following stipulation:
 \begin{equation}
 \label{equation:central}
 \gamma_J(x) = [\Delta x,\nabla x], \,\,\,\mbox{ for all } x\in J.
 \end{equation}
 Then $\gamma_J$  maps $J$ {\em one-one} into 
 $\mathcal{I}(\mathcal C(J))$.

  \end{itemize}
 \end{proposition}
 
 \begin{proof}  (i) 
 easily follows
  from \eqref{equation:deltadelta},\eqref{equation:deltanabla}
  and  \eqref{equation:quasi}.

  (ii) The closure of
$\mathcal C(J)$ under   $\neg,\oplus, \odot$ 
follows 
from  
(i) and \eqref{equation:nabla}, \eqref{equation:deltaoplus},
\eqref{equation:deltaodot}.
Thus, 
$\mathcal C(J)=(C(J),0,1,\neg,\oplus,\odot)$ is
a $(0,1,\neg,\oplus, \odot)$-subreduct   of $J$ and, as such,  it
necessarily
satisfies equations \eqref{equation:associative}--\eqref{equation:negneg}.
Further,  any two elements  
  $\alpha,\beta  \in \mathcal C(J)$ satisfy  the characteristic
  MV-algebraic equation  $\neg(\neg \alpha\oplus \beta)\oplus \beta
  =\neg(\neg \beta\oplus \alpha)\oplus \alpha.$ 
This is so  because by
Proposition \ref{proposition:varia}(i) we can write
$\alpha=\Delta x,\,\,\, \beta =\Delta y$ for suitable  $x,y\in J$,
and  $J$ satisfies  \eqref{equation:mangani}. 
Finally, since the  constant $1$ and the operation $\odot$
of $\mathcal C(J)$  are definable via  \eqref{equation:one}
and \eqref{equation:odot},  then $\mathcal C(J)$ is an
MV-algebra.

(iii) Easy.

(iv)  As already noted, every element 
 $x\in J$ satisfies equation \eqref{equation:order} stating that
 $(\Delta x,\nabla x)$ is a monotone
 pair of  $\mathcal C(J).$  Equation 
  \eqref{equation:quasi} ensures that different
   elements of $J$ are mapped by 
   $\gamma_J$ into different elements.
 \end{proof}

The following converse
of Proposition \ref{proposition:trivial}(i) shows that
$\gamma_J$  is an isomorphism of $J$ {\it onto}
the IMV-algebra of all intervals in  $\mathcal C(J).$

\medskip

\begin{theorem}[\bf Representation theorem for IMV-algebras]
\label{theorem:representation} 
Let $J$ be an IMV-algebra,  
$\mathcal C(J)$ its
central  MV-algebra, and
$\mathcal I(\mathcal C(J)) =
\{[\alpha,\beta ]\in C(J)\times C(J) \mid \alpha \leq \beta\}$
the
IMV-algebra of all intervals in 
 $ \mathcal C(J)$ equipped with the
 operations  \eqref{constants}--\eqref{nabla}.
 Then the map 
$
\gamma_J \colon x\in J\mapsto
[\Delta x,\nabla x] \in  \mathcal I(\mathcal C(J)) 
$
of Proposition \ref{proposition:varia}(iv)
is an isomorphism of $J$ onto
$\mathcal I(\mathcal C(J))$.
\end{theorem}

\begin{proof}
Evidently, $\gamma_J(0)=[0,0],\,\,
 \gamma_J(1)=[1,1],\,\,\gamma_J(\iii)=C(J)$.
By  \eqref{equation:deltaoplus} and  Proposition~\ref{proposition:riesz},
$
\gamma_J(x\oplus y)=[\Delta (x\oplus y),\nabla(x\oplus y) ]
=
[\Delta x\oplus\Delta y , \nabla x \oplus \nabla y]
=
[\Delta x, \nabla x]\oplus
[\Delta y, \nabla y]
=
 \gamma_J(x)\oplus \gamma_J(y).
$
By \eqref{equation:negneg} and
\eqref{equation:nabla},
  $\gamma_J(\neg x)=\neg\gamma_J(x).$  
By   \eqref{equation:deltaodot},  
 $\gamma_J(x\odot y)=\gamma_J(x)\odot \gamma_J(y).$ 
By \eqref{equation:deltadelta},
\eqref{equation:deltanabla} and \eqref{equation:nabla}
for all  $x\in J$ we have   
$\gamma_J(\Delta x)=[\Delta\Delta x, \nabla\Delta x]=[\Delta x, \Delta x]
=\Delta[\Delta x,\nabla x]=\Delta\gamma_J(x).$  
 Similarly,
$\gamma_J(\nabla x)=\nabla\gamma_J(x).$
Let $[\alpha,\beta]\in \mathcal I(\mathcal  C(J))$,\,\,\,
 $\gamma_J(\alpha\oplus \iii)=[\alpha,\alpha]\oplus[0,1]=[\alpha,1]$
 and   $\gamma_J(\beta\odot \iii)=[0,\beta]$.
Then $\gamma_J((\alpha\oplus \iii)\odot(\beta\odot \iii))=[\alpha,\beta]$.
Thus every interval in  $\mathcal C(J)$
is in $\gamma_J(J)$.  
\end{proof}

\section{IMV-algebras are categorically equivalent to MV-algebras}

In \eqref{equation:iota} we observed
 that the map $\iota_A\colon A\to \mathcal D(A)=\mathcal C(\mathcal I(A))$ is an isomorphism of MV-algebras.  Also, in Theorem~\ref{theorem:representation} we proved that the map $\gamma_{J}\colon J\to \mathcal I(\mathcal C (J))$ defined by $\gamma_{J}(x)=[\Delta(x),\nabla(x)]$ is an isomorphism of IMV-algebras. In this section we 
 prove that  the category $\MV$ of MV-algebras 
 with homomorphisms and the category  $\IMV$ 
 of IMV-algebras with homomorphisms are equivalent,
  and that $\iota$ and $\gamma$ are the natural isomorphisms 
  determining the 
   equivalence.
   
    For all unexplained notions in category theory
used in this paper, we refer to  \cite{mcl}.

\medskip
\begin{lemma}
\label{definition:mv-functors}
Let   $\mathcal I\colon \MV\to\IMV$ be  the assignment defined by:
\begin{eqnarray*}
\mbox{objects:} \,\,\,\,\,\,\,\,\,\,\,\,\,\,\,\,\,\,\,\,\,\,\,\,\,\, \,\,\,\,\,\,\,\,\,\,\,\,\,\,\,\,\,\,\,\, 
 A  \,\,\,\,&\mapsto&\,\,\,\, \mathcal I( A)\\
\mbox{morphisms:} \,\,\,\,\,\,\,\,\,\,\,\, \,\, h\colon A\to B\,\,\,\,& \mapsto&\,\,\,\,
 \mathcal I(h)\colon \mathcal I(A)\to\mathcal I(B),
\end{eqnarray*}

\medskip
\noindent
where the homomorphism  $\mathcal I(h)$
is given by $(\mathcal I(h))([\alpha,\beta])=[h(\alpha),h(\beta)]$, for any interval
$[\alpha,\beta]\in \mathcal I( A)$.
Conversely,  let 
  $\mathcal C\colon \IMV\to\MV$  be the assignment
defined by:
\begin{eqnarray*}
\mbox{objects:}  \,\,\,\,\,\,\,\,\,\,\,\,\,\,\,\,\,\,\,\,\,\,\,\,\,\,\,\,\,\,\,\,\,\,\,\,\,\,\,\,\,\,\,\,\,\,\,\,
\,\,\,\,\,\, \,\,\,
J \,\,\,\,&\mapsto& \mathcal C( J)\\
\mbox{morphisms:}\,\,\,\,\,\,\,\,\,\,\,\, \,\,\,\,\,\,\,\,\,\,\,
f\colon J\to K\,\,\,\,\,\,\, &\mapsto& \ \mathcal C( f)=
f\restrict{C(J)}.
\end{eqnarray*}
 Both $\,\,\mathcal I$ and $\mathcal C$ are  well-defined functors, respectively
called {\em interval} and {\em central} functor.
\end{lemma}

\begin{lemma}
\label{lemma:iota} 
The assignment $\iota\colon A \to \iota_A$ 
is a natural isomorphism from the identity 
functor ${\rm I}_{\MV}\colon \MV\to \MV$ 
to the composite functor $\mathcal{C}\circ \mathcal{I}$.
\end{lemma}

\begin{proof} 
By \eqref{equation:iota},
 for each MV-algebra $A$, the map $\iota_{A}$ is an MV-isomorphism. 
There remains to be proved
 that $\iota$ is a {\it natural transformation}, that is, 
for every MV-algebras $A,B$ and homomorphism
  $h\colon A\to B,$ the  diagram of Figure
  \ref{figure:sulcentro} commutes.

   \begin{figure} [ht]
\begin{center}
\begin{tikzpicture} 
[auto, text depth=0.25ex,] 
\matrix[row sep= .9cm, column sep= 1.9cm]
{
\node (11) {$A$}; &\node (12) {$B$};\\
\node (21) {$\mathcal C(\mathcal I(A))$}; &\node (22) {$\mathcal C(\mathcal I(B))$};\\
};
\draw [->] (11) to node {$h$} (12);
\draw [->] (21) to node [swap] {$\mathcal{C}(\mathcal{I}(h))$} (22);
\draw [->] (11) to node [swap] {$\iota_A$} (21);
\draw [->] (12) to node {$\iota_B$} (22);
\end{tikzpicture}
\end{center}
\caption{}
\label{figure:sulcentro}
\end{figure}

\noindent
Indeed, for each $\alpha \in A$ we can write
\[
\mathcal{C}(\mathcal{I}(h))(\iota_A(\alpha))
=\mathcal{C}(\mathcal{I}(h))([\alpha,\alpha])
=\mathcal{I}(h)([\alpha,\alpha])=[h(\alpha),h(\alpha)]
=\iota_{B}(h(\alpha)).\qedhere
\]
 \end{proof}
  
 \bigskip 
  \begin{lemma}
\label{lemma:gamma}
The assignment
 $\gamma\colon J \to \gamma_J$ is a natural isomorphism from the identity functor ${\rm I}_{\IMV}\colon \IMV\to \IMV$ to the composite functor $\mathcal{I}\circ \mathcal{C}$. 
\end{lemma}
\begin{proof}

By Theorem~\ref{theorem:representation}, for each IMV-algebra $J$ the map $\gamma_J\colon J\to \mathcal{I}(\mathcal{C}(J))$ is an isomorphism. 
To prove 
 that  $\gamma$ is a natural transformation,
  let  $J,K$ be IMV-algebras and $f\colon J\to K$ an IMV-homomorphism. Then the diagram of Figure \ref{figure:sulcentrobis} commutes.
   \begin{figure} [ht]
\begin{center}
\begin{tikzpicture} 
[auto, text depth=0.25ex,] 
\matrix[row sep= .9cm, column sep= 1.9cm]
{
\node (11) {$J$}; &\node (12) {$K$};\\
\node (21) {$\mathcal I(\mathcal C(J))$}; &\node (22) {$\mathcal I(\mathcal C(K))$};\\
};
\draw [->] (11) to node {$f$} (12);
\draw [->] (21) to node [swap]{$\mathcal{I}(\mathcal{C}(f))$} (22);
\draw [->] (11) to node [swap] {$\gamma_J$} (21);
\draw [->] (12) to node {$\gamma_K$} (22);
\end{tikzpicture}
\end{center}
\caption{}
\label{figure:sulcentrobis}
\end{figure}

\medskip
\noindent
Indeed,  since $f$ commutes with $\Delta$ and $\nabla$,
for each $x\in J$ we can write
\begin{align*}
\mathcal{I}(\mathcal{C}(f))(\gamma_{J}(x))&=\mathcal{I}(\mathcal{C}(f))([\Delta x,\nabla x])=[\mathcal{C}(f)(\Delta x),\mathcal{C}(f) (\nabla x)]=[f(\Delta x),f(\nabla x)]\\&=[\Delta f(x),\nabla f(x)]=\gamma_{K}(f(x)).\qedhere
\end{align*}

  \end{proof}

\bigskip
From Lemmas \ref{lemma:iota}
 and~\ref{lemma:gamma}, we obtain:
 
 \medskip
 
\begin{theorem}[\bf Categorical equivalence]
\label{theorem:equivalence} 
The  functors\,\,
 $\mathcal I$ and $\mathcal C$  and the natural isomorphisms 
$\iota$ and $\gamma$ determine  a categorical equivalence 
between MV-algebras and IMV-algebras.
\end{theorem}

\bigskip
\section{Completeness,  free IMV-algebras,  product and  inclusion order}
\label{section:completeness}
As a first application of the categorical equivalence  $\mathcal I$ between
MV-algebras and IMV-algebras, we give a complete
description of free IMV-algebras. Since categorical equivalence does not necessarilly preserve freeness, the description of  free {IMV}-algebras cannot be directly derived from a description of free {MV}-algebras.
%
%
%

  We refer to
\cite{cigdotmun} for all unexplained MV-algebraic notions
used here.  Fix an integer  $n>0$. Then
a  {\it McNaughton function} on $\I^n$ is a continuous piecewise linear
 map
$\xi\colon\I^n\rightarrow \I$ such that each linear piece of $\xi$
has integer coefficients.
 More generally, for any arbitrary (possibly infinite) set $X\not=\emptyset$,
a function $\eta\colon \I^{X}\to \I$ is  said to be a {\it McNaughton map}
if for some  finite $Y\subseteq X$, and  McNaughton function
$\xi\colon \I^{Y}\to \I$ we have

$$\eta=\xi\circ \pi^{X}_{Y},\,\,\,\mbox{where }\,\,\,
\pi^{X}_{Y}\colon \I^X\to \I^{Y}\,\,\,\mbox{is the projection map. }$$
\noindent
For each  closed set $S\subseteq \I^{X}$ we let 
$\McN(S)$ denote the set of restrictions to $S $
of the McNaughton functions on the cube $\I^{X}$, in symbols,
 $$\McN(S)=\{f\restrict S\mid f\colon \I^{X}\to \I\mbox{ a McNaughton map}\},$$
equipped with the pointwise operations of the standard MV-algebra
$\I.$  In this paper $S$ will always be nonempty, so that  
 $\McN(S)$ is a nontrivial MV-algebra.
 
 \medskip
 We let the triangle $\Theta\subseteq \I^2$ be defined by
 \begin{equation}
 \label{equation:triangle}
 \Theta=\{(\alpha,\beta)\in  \I^2  \mid \alpha\leq \beta \}.
 \end{equation}
 For each  $i=1,2$ we  also let $\pi_i\colon \Theta\to \I$
 be the projection maps  $\pi_i(\alpha_1,\alpha_2)=\alpha_i.$

\begin{theorem}[\bf Equational Completeness]
\label{theorem:hsp}
The operations of the IMV-algebra
$$
\boldsymbol{\Theta}=(\Theta,(0,0),(1,1),(0,1), \neg,\Delta,\nabla, \oplus,\odot)
$$
 are defined
 for each $(a,b),(c,d)\in \Theta$ as follows: 
 
 \smallskip
 \begin{itemize}
 \item[]  $\neg(a,b)=(\neg b,\neg a)$,\,\,\, $\Delta(a,b)
=(a,a)$, \,\,\,$\nabla(a,b)=(b,b)$, 

\medskip
\item[]  $(a,b)\oplus(c,d)=(a\oplus c,b\oplus d)$\,\, and 
\,\,$(a,b)\odot (c,d)=(a\odot c,b\odot d)$.
\end{itemize}

 \smallskip
 \noindent
Let the
 map  $\omega\colon \mathcal I([0,1]) \to \Theta$ be defined by
  $[a,b]\mapsto (a,b)$.  Then $\omega$ is an isomorphism between 
the IMV-algebras $\mathcal I([0,1])$ and  $\boldsymbol{\Theta}$, in symbols,
\begin{equation}
\label{equation:omega}
\omega\colon \mathcal I([0,1])\cong \boldsymbol{\Theta}. 
\end{equation}
Further, with the notation of \cite[Defintion~9.1]{bursan},
  \begin{equation}
  \label{equation:generators}
   \IMV=\mathbb{HSP}(\mathcal I([0,1]))=\mathbb{HSP}(\boldsymbol{\Theta}).
   \end{equation}
   Thus an equation holds in all IMV-algebras iff it holds in
   $\mathcal I([0,1])$. 
\end{theorem}

\begin{proof}
\eqref{equation:omega} is immediately verified in the light
of Proposition \ref{proposition:riesz}.
As a categorical equivalence between two varieties
(Theorem \ref{theorem:equivalence})
the interval functor $\mathcal I$ preserves
products, subalgebras and homomorphic images (the later 
are preserved since homomorphic images
in every variety  are codomains of
 regular epi-morphisms),  \cite{mcl}.
Using Chang completeness theorem
 \cite[Theorem~2.5.3]{cigdotmun} in combination with
 \cite[Theorem~9.5]{bursan}, we can write
  $\MV=\mathbb{HSP}([0,1])$. 
  Then
    \eqref{equation:generators} immediately follows
    from   Theorem~\ref{theorem:equivalence} and
  \eqref{equation:omega}.  The rest now follows from
  Birkhoff completeness theorem for  equational logic, \cite[Theorem~14.19]{bursan}.
\end{proof}

\begin{theorem} [\bf Free IMV-algebras]
\label{theorem:free}
For any cardinal 
$\kappa\geq 1$  the
  free  $\kappa$-generator IMV-algebra
  is the algebra $\mathcal{I}(\McN(\Theta^{\kappa}))$.
  A free generating set for this algebra is given by the
 intervals $[\pi_1\circ  \boldsymbol{\pi}_\alpha,\pi_2\circ  \boldsymbol{\pi}_\alpha ]$ where
  $  \boldsymbol{\pi}_\alpha\colon \Theta^{\kappa}\to \Theta$  
is the projection map into the $\alpha$th
coordinate,  for each ordinal $\alpha<\kappa$.
\end{theorem}

\begin{proof}
Let  $X$ be a set of cardinality $\kappa$. 
As a consequence of \eqref{equation:generators}, 
the free IMV-algebra on $\kappa$ generators is isomorphic to the IMV-subalgebra 
$\F_{\IMV}(X)$
of $\boldsymbol{\Theta}^{(\Theta^X)}$ generated by the projection maps. 

\medskip

\noindent{\it Claim 1}: Let $f\colon \Theta^X\to \Theta$ be an
 element of $\F_{\IMV}(X)$.
 Then both functions
 $\pi_1\circ f$ and $\pi_2\circ f$ are members of $ \McN(\Theta^{X})$.

\medskip
The proof is by induction on  the number of applications
of the  IMV-operations needed to obtain   $f$ from the
projection maps $ \boldsymbol{\pi}_x$.

\medskip
\noindent{\it Basis Step}: If
 for some $x\in X$,\,\,\,
  $f=  \boldsymbol{\pi}_x$, then the map $\pi_1\circ f=\pi_1\circ   \boldsymbol{\pi}_x
  \colon \Theta^X\to\I$ coincides over $\Theta^X$ with the map 
$  \boldsymbol{\pi}_{x,1}\colon (\I^2)^X\to \I$
sending  any $v\in (\I^2)^X$ into $\pi_1(v_x)$. 
 Since 
 $  \boldsymbol{\pi}_{x,1}$ is  continuous piecewise linear
  with only one piece,
   and its unique linear piece has integer coefficients,
    then $  \boldsymbol{\pi}_{x,1}$
 is a member of $
  \McN((\I^2)^X)$, and
  hence $f\in \McN(\Theta^X)$. Similarly, $\pi_2\circ \boldsymbol{\pi}_x\in \McN(\Theta^X)$.

\medskip
\noindent{\it Induction Step}: \,\,\, By \eqref{equation:negi}  and~\eqref{equation:one} it is enough to argue only for $\Delta$, $\neg$ and $\oplus$.
Let $f,g\colon \Theta^X\to \Theta$ be members of 
 $\F_{\IMV}(X)$.
For each $v\in \Theta^X$ we can write
\begin{align*}
(\pi_1\circ(\Delta_{\F_{\IMV}(X)} f))(v)&=\pi_1(\Delta  f(v)) =\pi_1(f(v))=(\pi_1\circ f)(v);\\
(\pi_2\circ(\Delta_{\F_{\IMV}(X)} f))(v)&=\pi_2(\Delta  f(v)) =\pi_1(f(v))=(\pi_1\circ f)(v);\\
(\pi_1\circ(\neg_{\F_{\IMV}(X)} f))(v)&=\pi_1(\neg  f(v)) =\neg(\pi_2(f(v)))=(\neg_{\McN(\Theta^X)}(\pi_2\circ f))(v);\\
(\pi_2\circ(\neg_{\F_{\IMV}(X)} f))(v)&=\pi_2(\neg  f(v)) =\neg(\pi_1(f(v)))=(\neg_{\McN(\Theta^X)}(\pi_1\circ f))(v)\\
(\pi_1\circ(f\oplus_{\F_{\IMV}(X)} g))(v)&=\pi_1( f(v)\oplus  g(v)) =\pi_1(f(v))\oplus \pi_1(g(v))\\
&=(\pi_1\circ f)\oplus_{\McN(\Theta^X)}(\pi_1\circ g)(c);\\
(\pi_2\circ(f\oplus_{\F_{\IMV}(X)} g))(v)&=\pi_2( f(v)\oplus  g(v)) =\pi_2(f(v))\oplus \pi_2(g(v))\\
&=\bigl((\pi_2\circ f)\oplus_{\McN(\Theta^X)}(\pi_2\circ g)\bigr)(c).
\end{align*}
Thus whenever $f$ and $g$ satisfy 
 $\pi_1\circ f,\,\,\,\pi_2\circ f,\,\,\,\pi_1\circ g,\,\,\,\pi_2\circ g\in\McN(\Theta^X)$, 
 all the maps 
 $\pi_1\circ(\Delta_{\F_{\IMV}(X)} f)$, \,\,\,$\pi_2\circ(\Delta_{\F_{\IMV}(X)} f)$, \,\,\,$\pi_1\circ(\neg_{\F_{\IMV}(X)} f)$, \,\,\,$\pi_2\circ(\neg_{\F_{\IMV}(X)} f)$, \,\,\,
 $\pi_1\circ(f\oplus_{\F_{\IMV}(X)} g)$, \,\,\, $\pi_2\circ(f\oplus_{\F_{\IMV}(X)} g)$  belong to $\McN(\Theta^X)$.

\medskip

\noindent{\it Claim 2}: $\mathcal{C}(\F_{\IMV}(X))\cong \McN(\Theta^{X})$.

Let $\eta\colon \mathcal{C}(\F_{\IMV}(X))\to \McN(\Theta^{X})$ be defined by $\eta(f)=\pi_1\circ f$ for any map
 $f\colon \Theta^{X}\to \Theta$  in  $\mathcal{C}(\F_{\IMV}(X))$. 
By Claim 1, 
 $\eta(f)=\pi_1\circ f\in\McN(\Theta^{X})$ for each $f\in\mathcal{C}(\F_{\IMV}(X))$, 
thus proving that  $\eta$ is well defined. 
To see that $\eta$ is an MV-homomorphism, first observe that $0_{\F_{\IMV}(X)}$ is the constant map $0_{\F_{\IMV}(X)}(x)=(0,0)$ for each $x\in \Theta^{X}$. Thus  $\eta(0_{\F_{\IMV}(X)})(x)=0$ for each $x\in \Theta^{X}$. 
Since $f=(f_1,f_2)\in \mathcal{C}(\F_{\IMV}(X))$,  
then 
$$f_1=\pi_1\circ f=\eta(f)=\eta(\nabla f)=\pi_1\circ (\pi_2\circ f,\pi_2\circ f)= \pi_2\circ f=f_2,$$  and for every $x\in \Theta^{X}$,
\begin{align*}
\eta(\neg_{\F_{\IMV}(X)} f)(x)&=\pi_1((\neg_{\F_{\IMV}(X)} f)(x))= \pi_1(\neg (f(x)))\\
&=\pi_1(\neg f_1(x),\neg f_1(x))=\neg f_1(x)=(\neg \eta(f))(x).
\end{align*}
Now let $f=(f_1,f_2),\,\,g=(g_1,g_2)\in \mathcal{C}(\F_{\IMV}(X))$. It follows that 
\begin{align*}
(\eta(f\oplus_{\F_{\IMV}(X)} g))(x)&=\eta( (f\oplus  g)(x))=\eta( (f_1(x)\oplus g_1(x),f_2(x)\oplus g_2(x)))\\
&=f_1(x)\oplus g_1(x)=(\eta(f)\oplus \eta(g))(x).
\end{align*}

\noindent
It is easy to see that
 $\eta$ is one-to-one. Indeed, $\eta(f)=\eta(g)$, implies  $\pi_1\circ f=\pi_1\circ g$ and $\pi_2\circ f=\pi_2\circ g$, that is, $f=g$. To see that $\eta$ is onto, for each $x\in X$ let $  \boldsymbol{\pi}_{1,x},  \boldsymbol{\pi}_{2,x}\colon(\I^2)^X\to \I^2$ denote the projection maps $  \boldsymbol{\pi}_{1,x}(v)=\pi_1(  \boldsymbol{\pi}_x(v))$ and  $  \boldsymbol{\pi}_{2,x}(v)=\pi_2(  \boldsymbol{\pi}_x(v))$. Since
 the MV-algebra $\McN((\I^2)^X)$ is generated by these projection maps, then $\McN(\Theta^X)$ is generated by the restriction of the projection maps $  \boldsymbol{\pi}_{1,x}$ and $  \boldsymbol{\pi}_{2,x}$ to $\Theta^X$. 
 Now for each $x\in X$ we have the identities
 $  \boldsymbol{\pi}_{1,x}\restrict{\Theta^X}
 =\pi_1\circ(\Delta_{\F_{\IMV}(X)}   \boldsymbol{\pi}_x)$  
 and $  \boldsymbol{\pi}_{2,x}\restrict{\Theta^X}
 =\pi_1\circ(\nabla_{\F_{\IMV}(X)}   \boldsymbol{\pi}_x)$. Further,
  $\Delta_{\F_{\IMV}(X)}\boldsymbol{\pi}_x,\nabla_{\F_{\IMV}(X)} 
    \boldsymbol{\pi}_x\in \mathcal{C}(\F_{\IMV}(X))$.
 As a consequence,   $\eta$ is onto $\McN(\Theta^X)$.

 \medskip

From Claim 2 and Theorem~\ref{theorem:representation}, it follows that 
\[
\F_{\IMV}(X)\cong\mathcal{I}(\mathcal{C}(\F_{\IMV}(X)))
\cong \mathcal{I}(\McN(\Theta^X)), 
\]
whence we can write 
  $\mu\colon \F_{\IMV}(X)\cong \mathcal{I}(\McN(\Theta^X))$  for
  some one-one onto map.  For every $x\in X$,
$\mu(  \boldsymbol{\pi}_x)=
[\eta(\Delta_{\F_{\IMV}(X)}  \boldsymbol{\pi}_x),
\eta(\nabla_{\F_{\IMV}(X)}  \boldsymbol{\pi}_x)]
=[\pi_1\circ  \boldsymbol{\pi}_x,\pi_2\circ  \boldsymbol{\pi}_x].
$
\hfill{\qedhere}
\end{proof}

\medskip

   \subsection*{The product order in IMV-algebras}
  Just as every MV-algebra  $A$ has an underlying lattice structure that
  is definable from the monoidal structure of $A$,  also every IMV-algebra 
  has the following (product)  lattice order.

  \begin{proposition}
  \label{proposition:positional}
  Let $J$ be an IMV-algebra, identified with $\mathcal I(\mathcal C(J))$
  via the isomorphism  $\gamma_J$ of 
  Theorem \ref{theorem:representation}. For any two intervals
  $x,y\in J$  let us write  without fear of ambiguity $x=[\alpha,\beta],
  \,\,\,y=[\gamma,\delta]$ for uniquely determined elements
  $\alpha,\beta,\gamma,\delta$ of the MV-algebra $\mathcal C(J).$
  Let us define
  \begin{equation}
  \label{equation:central-order}
u \boldsymbol{\wedge} v =\neg(\neg u \odot v)\odot v
\,\,\,\,\mbox{  and  }\,\,\,\,
u \boldsymbol{\vee} v =\neg(\neg u \oplus v)\oplus v, \mbox{ for all }  u,v\in J.
  \end{equation}
We then have:
  \begin{itemize}
  \item[(i)] 
  Over the MV-algebra  $\mathcal C(J)$
the operations
  $\boldsymbol{\wedge,\vee}$  coincide with  the  lattice operations $\wedge,\vee$  of
  the
MV-algebra $\mathcal C(J).$
  
 \smallskip 
    \item[(ii)]
  Upon writing
  $
  x\sqcap y= [\alpha{\wedge} \gamma,\beta{\wedge} \delta]
  $
  and
  $
  x\sqcup y= [\alpha{\vee} \gamma,\beta{\vee} \delta]
  $,
  the algebra  
  \begin{equation}
  \label{equation:capcup}
  J^*=(J,0,1,\iii,\neg,\Delta,\nabla,\oplus,\odot,
  \sqcup,\sqcap)
  \end{equation}
has  a distributive lattice reduct  $(J,0,1,  \sqcup,\sqcap)$
  with largest  element $1$ and smallest $0$.

  \smallskip
  \item[(iii)]
  Denoting by $\sqsubseteq$ the resulting partial order on
  $J^*$, it follows that $\oplus$ and $\odot$ are monotone in both
arguments,  $\neg$ is order-reversing,   and
$\Delta x\sqsubseteq x \sqsubseteq \nabla x.$

 \smallskip 
  \item[(iv)]
  Generalizing the definition of the left-hand term of \eqref{equation:great}, let
 the binary IMV-term
 $\zeta(u,v)$   be defined by
 \begin{equation}
 \label{equation:zeta}
\zeta(u,v)=\Delta u\oplus(\iii\odot \nabla v \odot\neg \Delta u),
\,\,\,\mbox{ for all }  u,v\in J.
\end{equation}
Suppose  $u=[\alpha,\alpha]$ and
$v=[\beta,\beta]$ belong to    $\mathcal C(J)$,
 and $u\sqsubseteq v$,
 which in the present case is 
  equivalent to  $\alpha\leq \beta$ in the MV-algebra
  $\mathcal C(J)$. Then 
   $\Delta\zeta(u,v)=u$ and $\nabla\zeta(u,v)=v$, that is,
 $$\zeta(u,v)=[\alpha,\beta].$$
 
  \smallskip 
    \item[(v)] 
  The lattice operation $\sqcap$ of  $J^*$ is definable from the
  operations of $J$, by
  $$
  x\sqcap y =
  \zeta(\Delta x\boldsymbol{\wedge} \Delta y, \nabla x \boldsymbol{\wedge} \nabla y).
  $$

    \smallskip 
      \item[(vi)] 
  The lattice operation $\sqcup$ of  $J^*$ is definable from the
  operations of $J$, by
  $$
  x\sqcup y =
  \zeta(\Delta x\boldsymbol{\vee} \Delta y, \nabla x \boldsymbol{\vee} \nabla y).
  $$
  
      \smallskip 
      \item[(vii)]   For all $x,y\in J$ we have
      $x\sqsubseteq y$  iff $x\sqcup y=y$ iff
      $(\neg \Delta x\oplus \Delta y)\odot 
      (\neg \nabla x\oplus \nabla y)=1.$
 
  \end{itemize}
  \end{proposition}
  
  \begin{proof}
  Routine.
  \end{proof}
  
  \begin{remark}
  In the IMV-algebra  $J$ the  operations
  $\boldsymbol{\wedge}$
    defined 
     in \eqref{equation:central-order}
     and 
     $\sqcap$
     defined  in  (v) above
     do not coincide in general.
     Similarly,    $\boldsymbol{\vee}$ 
     need not coincide with~$\sqcup$.
  \end{remark}

\medskip
In  \cite{gehwal} the authors  study  the algebra
(denoted  $I^{[2]}$)  given by the
$(0,1,\Delta,\nabla,\sqcup,\sqcap)$-reduct of
  $(\mathcal I(\I))^*$
   as defined in \eqref{equation:capcup}. Thus in $I^{[2]}$, 
  $[\alpha,\beta] \sqsubseteq  [\alpha',\beta']$ means  $\alpha\leq \alpha'$
and $\beta\leq \beta'$. 
A commutative
associative  operation
$\bigcirc \colon \mathcal I(\I)\times 
 \mathcal I(\I) \to \mathcal I(\I)$ is said to be  
 a {\it t-norm} if it satisfies the following conditions:
 
\medskip
 \begin{itemize}
 \item[(a)]  $\mathcal C(\mathcal I(\I))\bigcirc \mathcal C(\mathcal I(\I))
 \subseteq \mathcal C(\mathcal I(\I))$,
 
\medskip
 \item[(b)]  $\bigcirc $ distributes over the   
 lattice operations  $\sqcup,\sqcap$  of   $I^{[2]}$,
 
\medskip
 \item[(c)]    for all $x\in   \mathcal I(\I)$,
 $1\bigcirc x=x,\,\,\, \iii\bigcirc [\alpha,\beta]=[0,\beta].$
 \end{itemize}

\medskip
\noindent
In addition,  $\bigcirc$  is said to be {\it convex}  if 
 whenever  $x\in  \mathcal I(\I)$  satisfies $\alpha\bigcirc \alpha'
  \sqsubseteq x \sqsubseteq
 \beta\bigcirc\beta'$,  it follows that
 $x=\epsilon \bigcirc\delta$ for some
 $\alpha\leq \epsilon \leq \beta$  and
 $\alpha'\leq \delta\leq  \beta'.$
 
 \begin{theorem}
The $\odot$ operation of  $\mathcal I(\I)$
 equips the lattice
$I^{[2]}$ with a convex t-norm in the sense of \cite{gehwal}.
\end{theorem}
\begin{proof}
 Use  Propositions \ref{proposition:riesz} and
 \ref{proposition:positional}.
\end{proof}

\subsection*{The inclusion order in IMV-algebras}
Beyond the  lattice order  $\sqcap,\sqcup$,
every IMV-algebra  $J$ is equipped with
a partial order relation  $\subseteq,$
given by the inclusion relation between intervals
of  $J=\mathcal I(\mathcal C(J)).$
Notwithstanding the categorical equivalence
between IMV-algebras and MV-algebras, the
inclusion order has no role in MV-algebras.

  \begin{proposition}
  \label{proposition:inclusion}
  Adopt the  hypotheses and notation of
Proposition
  \ref{proposition:positional}.
  We then have:
  \begin{itemize}
  \item[(i)]  
  Upon writing
  $
  x\cup y= [\alpha\wedge \gamma,\beta\vee \delta],
  $
  the algebra  
  $$J^{**}=(J,0,1,\iii,\neg,\Delta,\nabla, \oplus, \odot, \cup)$$
  becomes  a sup-semilattice  
  with maximum element $\iii$,  whose set of minimal elements
coincides with the center of $J$.

  \smallskip
  \item[(iii)]
  Denoting by $\subseteq$ the resulting partial order on
  $J^{**}$, it follows that $\oplus$ and $\odot$ are monotone in both
arguments,  $\neg$ is monotone,   and
$\Delta x\subseteq x \supseteq \nabla x.$

 \smallskip 
  \item[(iv)] The sup-semilattice operation  $\cup$ of 
    $J^{**}$ is definable  from the
  operations of $J$, by
  $
x\cup y =  \zeta (\Delta x \wedge \Delta y, \nabla x\vee \nabla y)
=
  [\Delta x \wedge \Delta y, \nabla x \vee \nabla y].
  $

    \smallskip 
      \item[(vi)]  For all $x,y\in J$ we have
      $x\subseteq y$  iff $x\cup y=y$ iff
       $(\Delta y\leq \Delta x)\mbox{ and }(\nabla x\leq \nabla y)$
       in the MV-algebra  $\mathcal C(J)$
       iff 
       $ (\neg\Delta y\oplus \Delta x)\odot(\neg\nabla x\oplus \nabla y)=1$
       in $J$.
 
  \end{itemize}
 \end{proposition} 
 
 \begin{proof}
 From  Proposition  \ref{proposition:positional}.
 \end{proof}
 
 \begin{remark}
Going back to the outset of  this paper, it is interesting
to note  that  IMV-algebras, while
categorically equivalent to MV-algebras,
can express the fundamental inclusion order  $u\subseteq v$
between actual measurements/estimations $u,v$, 
meaning that $u$ is more precise than $v$.  Within
the MV-algebraic framework, the inclusion order has no meaning,
because truth-values in \luk\ logic are real numbers, corresponding
to error-free normalized measurements. 
 \end{remark}

\bigskip

\section{The equational theory of IMV-algebras is coNP-complete}
We refer to 
\cite{garjoh} for algorithmic complexity theory.

Let $\mathcal X$ be a countable set, whose elements are
called {\it variables}. 
For any finite subset $X=\{X_1,\dots, X_n\}$ of $X$ we denote
by MV$(X)$  (resp., IMV(X)) the set of MV-terms  (resp., IMV-terms)
 $\omega$ such that all variables
occurring in $\omega$ belong to  $X$. Letting
$X$ range over all finite subsets of $\mathcal X$ we obtain
the set  MV$(\mathcal X)$
(resp., IMV$(\mathcal X)$)  of MV-terms  (resp., IMV-terms)
over $\mathcal X$.  
The {\it equational theory}  Id$_{\rm MV(\mathcal X)}$
of MV-algebras
(resp., Id$_{\rm IMV(\mathcal X)}$ of IMV-al\-ge\-b\-ras) is the
set of all equations over $\mathcal X$ satisfied by every MV-algebras
(resp., every IMV-algebra).
In the jargon of algorithmic complexity theory, the set 
Id$_{\rm MV(\mathcal X)}$ (resp., Id$_{\rm IMV(\mathcal X)}$)
 amounts to the following   ``problem'': 

\medskip
\noindent ${\mathsf{INSTANCE}:}$   Two  MV-terms  (resp., two IMV-terms) 
$\sigma_1$ and $\sigma_2$ in the variables $X_1,\ldots,X_n,\,\,\,(n=1,2,\dots)$.

 \smallskip
\noindent ${\mathsf{QUESTION}:}$ Does the equation
$\sigma_1=\sigma_2$ hold in every MV-algebra  (resp., in every IMV-algebra)?

\bigskip
The equational theory of MV-algebras
has the same algorithmic complexity
as the set of tautologies (in  the variables of $\mathcal X$)  of infinite-valued
\luk\ logic, whence it 
is {\it coNP-complete} by  \cite[Theorem~9.3.8]{cigdotmun}.
A similar result holds   for IMV-algebras.

 \begin{theorem}
 \label{theorem:equational}
The equational theory 
${\rm Id}_{\rm IMV(\mathcal X)}$ 
 is coNP-complete.
 \end{theorem}

 \begin{proof}
 We first describe a polytime reduction  $\psi\colon
 (\sigma_1,\sigma_2)\mapsto (\sigma'_1,\sigma'_2)$
  of   the equational theory
  Id$_{\rm MV(\mathcal X)}$
 of MV-algebras  to the equational theory
 Id$_{\rm IMV(\mathcal X)}$. Using
  Chang's distance function, \cite[\S 1]{cigdotmun}, 
 it is no loss of generality to assume   $\sigma_2=0.$  So 
 $\psi$  will transform any given  MV-term  
 $\sigma_1=\sigma(X_1,\dots,X_n)$
  into an
 IMV-term  $\sigma'$  such that $\mathsf{MV}\models \sigma=0$
 iff $\mathsf{IMV}\models \sigma'=0$.  Here, as usual,
 the symbol $\models$  stands for the tarskian  satisfaction relation
  of first-order logic, \cite[pp. 71 and 195]{bursan}.  Let $\I$ be the
 standard MV-algebra.
  We have the following equivalences
 \begin{eqnarray*}
&&\mathsf{MV}\models \sigma(X_1,\dots,X_n)=0\\
 &\Leftrightarrow&  \I\models \sigma(X_1,\dots,X_n)=0, \mbox{ by \cite[Theorem~2.5.3]{cigdotmun}}\\
  &\Leftrightarrow&\mathcal I(\I)\models \forall X_1\dots\forall X_n\in \mathcal C(\mathcal I(\I))\,\, \sigma(X_1,\dots,X_n)=0, \mbox{ by  Theorem \ref{theorem:equivalence} }\\
    &\Leftrightarrow&\mathcal I(\I)\models \forall X_1\dots\forall X_n\in \mathcal I(\I)\,\, \sigma(\Delta X_1,\dots,\Delta X_n)=0, 
    \mbox{ by Prop.~\ref{proposition:varia}(i)}\\
 &\Leftrightarrow& \mathsf{IMV}\models \sigma(\Delta X_1,\dots,	\Delta X_n)=0,
 \mbox{ because} \,\,\, \IMV=\mathbb{HSP}(\mathcal I([0,1])), \mbox{ by }\eqref{equation:generators}. 
 \end{eqnarray*}
 Letting  $\sigma' =  \sigma(\Delta X_1,\dots,	\Delta X_n)$, a direct inspection shows
 that the map $\sigma\mapsto \sigma'$  can be computed in deterministic polynomial time.
This provides the desired reduction  $\psi$ of 
Id$_{\rm MV(\mathcal X)}$  to Id$_{\rm IMV(\mathcal X)}$. 
As already noted,
 the equational theory of MV-algebras is coNP-complete, whence
 the equational theory of IMV-algebras is coNP-hard. 
 
\smallskip
In order to prove that  Id$_{\rm IMV(\mathcal X)}$  is in coNP
we will construct a polytime reduction
$\chi$  of 
 Id$_{\rm IMV(\mathcal X)}$ to  the consequence problem
 in \luk\ logic, \cite[\S 18]{mun11}.
 Given an equation  $\omega=\sigma,$  for IMV-terms
$\omega, \sigma$ in the variables $X_1,\dots,X_n, \,\,\,
(n=1,2,\dots),$
we prepare   
 $2n$  new variables  $Y_1,\dots,Y_n,Z_1,\dots, Z_n,$
and  write down the following equivalences:
 \begin{eqnarray*}
 \mathsf{IMV}\models  \omega=\sigma  &\Leftrightarrow&\mathcal I(\I)\models  \omega=\sigma, \mbox{  by    \eqref{equation:generators}}\\
     &\Leftrightarrow&\mathcal I(\I)\models \Delta\omega=\Delta\sigma\,\, \mbox{ and }\,\,
    \mathcal I(\I)\models \nabla\omega=\nabla\sigma,  \mbox{ by \eqref{equation:quasi}}.
  \end{eqnarray*}

We now focus on the first leg 
\begin{equation}
\label{equation:instance}
\mathcal I(\I)\models \Delta\omega=\Delta\sigma
\end{equation}
of the last equivalence.
 
In every IMV-algebra,
repeated application of  the equations
\eqref{equation:associative}--\eqref{equation:great}  
shows that
the IMV-term
  $\Delta\omega(X_1,\dots,X_n)$
  can be transformed in polytime into
an  equivalent  ``normal form''
   IMV-term
$
  \omega'(\Delta X_1,\dots,\Delta X_n,\nabla X_1,\dots,\nabla X_n),
$
  where the constant $\iii$ does not appear, the
  $\Delta,\nabla$ symbols only occur immediately before variables, and the total
  number of symbols in $\omega'$ is proportional to that of $\omega.$  
Replacing in $\omega'$  every occurrence of  $\Delta X_i$
  by the variable $Y_i$ and every occurrence of $\nabla X_i$
  by $Z_i$\,, we obtain the IMV-term  
  \begin{equation}
  \label{equation:normalform}
  \omega^{\Delta}(Y_1,\dots,Y_n,Z_1,\dots,Z_n)
  \end{equation}
 {\it which is also readable as an MV-term}. Similarly there is
 a polytime transformation  
   \begin{equation}
  \label{equation:normalform-bis}
  \sigma(X_1,\dots,X_n) \mapsto \sigma^{\Delta}(Y_1,\dots,Y_n,Z_1,\dots,Z_n).
    \end{equation}
  By construction, 
   \begin{eqnarray*}
  \mathcal I(\I)\models \Delta\omega=\Delta\sigma
      &\Leftrightarrow&\mathcal I(\I)\models 
      \omega' =
      \sigma' \mbox{ for all } X_1,\dots,X_n\\
      &\Leftrightarrow& \I\models \omega^{\Delta}  =\sigma^{\Delta}
      	\mbox{ whenever }Y_1\leq Z_1,\dots,
      Y_n\leq Z_n\\
         &\Leftrightarrow&Y_1\leq Z_1,\dots,
      Y_n\leq Z_n\vdash_{ {\mbox \L}_\infty} \omega^{\Delta}  =\sigma^{\Delta},   
     \end{eqnarray*} 
     where  $\vdash_{ {\mbox \L}_\infty}$ 
     denotes  consequence in \luk\ logic.
     
     \smallskip
     Turning to the second leg
     $
     \mathcal I(\I)\models \nabla\omega=\nabla\sigma,
     $
      we similarly reduce it to an equivalent instance
     $
     Y_1\leq Z_1,\dots,
      Y_n\leq Z_n\vdash_{ {\mbox \L}_\infty} \omega^{\nabla}  =\sigma^{\nabla}
     $
     of the consequence problem in \luk\ logic, for suitable
    polytime computable  MV-terms  $ \omega^{\nabla}$ and $\sigma^{\nabla}$.
     
     \smallskip
     
In conclusion,  the map
$
\chi\colon (\omega,\sigma)\mapsto 
(Y_1, Z_1,\dots,
      Y_n, Z_n,  \omega^{\Delta}, \sigma^{\Delta}, \omega^{\nabla}, \sigma^{\nabla})
      $
determines    a polytime reduction of
       Id$_{\rm IMV(\mathcal X)}$ to the consequence problem in \luk\ logic.  
A refinement of the proof
of \cite[Theorem~18.3]{mun11} in the light of 
 \cite[Theorem 9.3.4]{cigdotmun} shows that   this latter problem is in coNP,
 whence so is the problem   Id$_{\rm IMV(\mathcal X)}$.
      \end{proof}

 \section{Tautologies and consequence in \luk\ interval logic}
 \label{section:logic}
The proof of Theorem \ref{theorem:hsp} 
routinely yields a transformation
of IMV-equational logic into a deductive algorithmic
machinery on IMV-terms, which we will call ``\luk\ interval logic''.
One  first says that an IMV-term $\sigma(X_1,\dots,X_n)$
is an {\it IMV-tautology} if the equation $\sigma=1$ is
satisfied by all IMV-algebras. We call $\sigma$
an {\it $\mathcal I(\I)$-tautology} if the
 IMV-algebra  $\mathcal I(\I)$ satisfies equation
$\sigma=1$.
By Theorem \ref{theorem:hsp},  IMV-tautologies 
coincide with $\mathcal I(\I)$-tautologies, and will
be called {\it tautologies} without fear of confusion.

Let  $\mathsf{FORM}(X_1,\dots,X_n)$ be
the (absolutely free) algebra of IMV-terms
in the variables $X_1,\dots,X_n$.
Given   $\theta_1,\dots,\theta_m,\psi\in \mathsf{FORM}(X_1,\dots,X_n)$
 we say
that $\psi$ is a {\it consequence }  of
$\theta_1,\dots,\theta_m$, in symbols,
$$
\theta_1,\dots,\theta_m\vdash_{\rm IMV}\psi
$$
if  every homomorphism
(valuation, evaluation, truth-value assignment, interpretation, model)  
$
\eta\colon \mathsf{FORM}(X_1,\dots,X_n)\to \mathcal I(\I)
$ 
that evaluates to 1 each $\theta_i$, also satisfies
$\eta(\psi)=1.$
More generally, when $\Phi$ is an infinite set of
IMV-terms, we write $\Phi\vdash_{\rm IMV}\psi$
iff  $\Phi'\vdash_{\rm IMV}\psi$ for some finite
$\Phi'\subseteq \Phi.$ 
In particular, the notation
\begin{equation} 
\label{equation:tautology}
\emptyset \vdash_{\rm IMV}\psi
\end{equation}
precisely means that  $\psi$ is a tautology.

In \luk\ interval logic one can reduce the consequence
problem to the tautology problem, because of the
following counterpart of the ``local deduction theorem''
of \luk\ logic,

\begin{theorem}
\label{theorem:local}
For any
 $\theta_1,\dots,\theta_m,\psi\in \mathsf{FORM}(X_1,\dots,X_n)$
 the following conditions are equivalent: 
 \begin{itemize}
 \item[(i)] $\theta_1,\dots,\theta_m\vdash_{\rm IMV}\psi$.
 
 \medskip
  \item[(ii)] For some integer $k>0$ we have a tautology
  $$
 \neg\underbrace{( \Delta \theta_1\odot\dots\odot \Delta \theta_m)\oplus\dots
 \oplus \neg ( \Delta \theta_1\odot\dots\odot \Delta \theta_m)}_{ k
\,\, \mbox{\rm \tiny times}} \oplus \Delta \psi.
  $$
  
   \medskip
  \item[(iii)] For some integer $k>0$ we have a tautology
  $$
 \underbrace{( \Delta \theta_1\odot\dots\odot \Delta \theta_m)\odot \dots
 \odot  ( \Delta \theta_1\odot\dots\odot \Delta \theta_m)}_{ k
\,\, \mbox{\rm \tiny times}} \too\Delta \psi,
$$
where  the {\em implication interval connective}  $\too$ is defined
by $u\too v=\neg u\oplus v.$

   \medskip
  \item[(iv)]  $\Delta\psi$  is obtainable via Modus Ponens
  from $\Delta \theta_1\odot\dots\odot \Delta \theta_m$
  and the tautologies.
 \end{itemize}
\end{theorem}
\begin{proof}
One first notes that  a homomorphism $\eta$ of 
$\mathsf{FORM}(X_1,\dots,X_n)$ into $\mathcal I(\I)$
evaluates to 1 an IMV term $\sigma$  iff 
$\eta(\Delta \sigma)=1.$ So without loss of
generality the consequence relation
$\vdash_{\rm IMV}$ deals with interpretations of
IMV-terms in the center of  $\mathcal I(\I)$---which
we can safely identity with $\I$, in the light of Proposition
\ref{proposition:trivial}(ii).
Letting  $\vdash_{L_\infty}$ denote consequence in \luk\ logic,
\cite[\S 1]{mun11},
we can write:
\begin{eqnarray*}
&{}&\theta_1,\dots,\theta_m\vdash_{\rm IMV}\psi\\[0.2cm]
&\Leftrightarrow&\Delta\theta_1,\dots,\Delta\theta_m\vdash_{\rm IMV}
\Delta\psi\\[0.2cm]
&\Leftrightarrow&\theta_1^\Delta,\dots,\theta_m^\Delta\vdash_{\mbox{\L}_\infty}
\psi^\Delta\,\,\,\, \mbox{where the 
IMV-terms $\theta_i^ \Delta$ and $\psi^\Delta$ are readable}  \\[0.05cm]
&&\mbox{ as
\L$_\infty$-formulas
via the {\it polytime}  transformation  $\sigma\mapsto\sigma^{\Delta}$ of
\eqref{equation:normalform}-\eqref{equation:normalform-bis} }\\[0.2cm]
&\Leftrightarrow&\emptyset \vdash_{\mbox{\L}_\infty}
 \neg\underbrace{(\theta_1^\Delta \odot\dots\odot  \theta_m^\Delta)
 \oplus\dots\oplus \neg (\theta_1^\Delta \odot\dots\odot  \theta_m^\Delta)}_{ k
\,\, \mbox{\tiny \rm times}} \oplus  \psi^\Delta\\[0.2cm]
&\Leftrightarrow&\emptyset \vdash_{\rm IMV}
 \neg\underbrace{( \Delta \theta_1\odot\dots\odot \Delta \theta_m)\oplus\dots\oplus \neg ( \Delta \theta_1\odot\dots\odot \Delta \theta_m)}_{ k
\,\, \mbox{\tiny \rm times}} \oplus \Delta \psi,
\end{eqnarray*}
for some positive integer $k$.  This follows
from  the {\it local deduction theorem } in \luk\ logic,
\cite[1.7]{mun11}.
To conclude the proof,  let us note that  the MV-term
$$
\neg\underbrace{(\theta_1^\Delta \odot\dots\odot  \theta_m^\Delta)\oplus\dots
 \oplus \neg (\theta_1^\Delta\odot\dots\odot \theta_m^\Delta)}_{ k
\,\, \mbox{\rm \tiny times}} \oplus  \psi^\Delta
$$
can be equivalently   rewritten as
$$
 \underbrace{(\theta_1^\Delta \odot\dots\odot  \theta_m^\Delta)\odot \dots
 \odot  (\theta_1^\Delta \odot\dots\odot  \theta_m^\Delta}_{ k
\,\, \mbox{\rm \tiny times}} \too \psi^\Delta\,.
$$
 A final application of 
\cite[Definition~1.7]{mun11}
yields the desired conclusion.
\end{proof} 

As an immediate consequence of  Theorem \ref{theorem:equational}
we have

\begin{corollary}
\label{corollary:conp}
The tautology problem in \luk\ interval logic is
coNP-com\-p\-le\-te, and so is the consequence problem
$
\theta_1,\dots,\theta_m\vdash_{\rm IMV}\psi
$
\end{corollary}

\begin{remark}  Proof-theoretically oriented readers may
envisage other approaches  to \luk\ interval logic,
beyond Theorem \ref{theorem:local}.
For instance,  one can obtain from 
Theorem \ref{theorem:hsp}
a  proof system (with {\it axioms} and {\it rules})
originating from the equational logic of IMV-algebras.
To this purpose, one first notes that the IMV-term
$\delta(x,y)= (x\odot \neg y)\oplus(y \odot \neg x)$
coincides with
Chang distance \cite[Definition~1.2.4]{cigdotmun} when interpreted
in the center  $\I$ of $\mathcal I(\I)$. 
Secondly,  one may  transform
 each IMV-axiom  $\omega=\sigma$ of the list
 \eqref{equation:associative}--\eqref{equation:great}
 into the tautology  $\overline{\omega=\sigma}$
 given by the IMV-term
 $ \neg\delta(\Delta\omega,\Delta\sigma)\odot \neg\delta(\nabla\omega,\nabla\sigma).
$
These nineteen tautologies are the  
 {\it axioms of \luk\ interval logic}.
Finally, whenever a rule 
 $$
 \frac{{\omega_1=\sigma_1},\dots,{\omega_m=\sigma_m}}{{\omega=\sigma}}
 $$
of equational logic is applied to obtain
 a new equation  $\omega=\sigma$ from old equations
 $\omega_1=\sigma_1,\dots,\omega_m=\sigma_m$,
 the corresponding {\it rule  of \luk\ interval logic}
 $$
 \frac{\overline{\omega_1=\sigma_1},\dots,\overline{\omega_m=\sigma_m}}{\overline{\omega=\sigma}}
 $$
  is applied  to derive
 $\overline{\omega=\sigma}$ from
  $\overline{\omega_1=\sigma_1},\dots,\overline{\omega_m=\sigma_m}$. 
  Since equational logic has finitely many rules (essentially: instantiation
  and congruence), then so does
the resulting proof system for \luk\ interval logic. 
A main reason of interest in the proof system of Theorem \ref{theorem:local}
is that Modus Ponens is its only rule.
 
As of today,  whatever  proof system  $\mathcal S$ 
one may choose  for
\luk\ interval logic,  $\mathcal S$ will 
unavoidably require exponential space.
This is a consequence of  Corollary \ref{corollary:conp}---in 
want
of  an answer to the P/NP problem.

\end{remark}
 
 \subsection*{A term-equivalent implicative reformulation of IMV-algebras}
 Just as MV-algebras have a term-equivalent variant \cite[\S 4]{cigdotmun} where
 the monoidal operations  $\oplus,\odot$ are replaced by
 \luk\ implication $x\too y=\neg x\oplus y$,  
 also IMV-algebras have a term-equivalent
 counterpart based on the operation $\too$
 of Theorem \ref{theorem:local}(iii).
 The main interest in this reformulation is in
 Theorem \ref{theorem:local}(iv), stating that
consequences in  \luk\ interval logic can
be computed using Modus Ponens
as the only rule, once formulas are
written using  $\too$  instead of the monoidal
connectives $\oplus,\odot.$

\begin{proposition}
\label{proposition:IIMV}
For any IMV-algebra  $J$
let  $(J,\too)$ denote  $J$ enriched
with the  $\too$ operation.

 \bigskip
(i)   Then $(J,\too)$ obeys the following
equations:
\begin{eqnarray}
\label{equation:associative-too}
x\too (y\too z)&=&y\too (x\too z) \\
\label{equation:commutative-too}
x\too y&=&\neg y \too \neg x\\
 \label{equation:opluszero-too}
 \neg x &=&x \too \Delta \iii\\
\label{equation:coneutral-too}
\neg\Delta \iii &=&\Delta \iii \too x\\
\label{equation:negneg-too}
\neg\neg x&=&x\\  
\label{equation:mangani-too}
(\Delta x\too \Delta y)\too \Delta y&=&
(\Delta y\too \Delta x)\too \Delta x\\
\label{equation:negi-too}
\neg \iii&=&\iii\\
\label{equation:deltadelta-too}
\Delta\Delta x&=&\Delta x\\
\label{equation:deltanabla-too}
\Delta\neg\Delta  x&=&\neg\Delta x\\
\label{equation:deltaoplus-too}
\Delta(\neg x\too y)&=&\neg\Delta x\too \Delta y\\
\label{equation:deltaodot-too}
\neg\Delta\neg(x\too y)&=&\Delta x\too \neg\Delta\neg y\\
\label{equation:order-too}
\Delta \neg x \too \neg\Delta  x &=& \neg\Delta \iii \\
\label{equation:great-too}
(\iii \too (\neg\Delta\neg x \too \Delta x))\too\Delta x &=&x.
\end{eqnarray}

\medskip
(ii)
Conversely, suppose an algebra  $L=(L,\iii,\neg, \Delta,\too)$
of type  $(0,1,1,2)$
 satisfies equations
  \eqref{equation:associative-too}--\eqref{equation:great-too}.
Define  $0=\Delta\iii,\,\,\, 1=\neg\Delta \iii,\,\,\, \nabla x=\neg\Delta\neg x,\,\,\,
x\oplus y=\neg x\too y,\,\,\, x\odot y=\neg(x\too \neg y)$.
Then the algebra  $(L,0,1,\iii,\neg, \Delta,\nabla,\oplus,\odot)$
is an IMV-algebra.

\medskip
(iii)
IMV-algebras are term-equivalent to the algebras satisfying
 the equations
  \eqref{equation:associative-too}--\eqref{equation:great-too}.

\medskip
(iv) If an equation  
 is satisfied the algebra  $(\I, \iii,\neg,\Delta,\too)$
then it is satisfied by all algebras
satisfying the equations \eqref{equation:associative-too}--\eqref{equation:great-too}.
 
\medskip
(v)  Thus an equation is obtainable from
equations \eqref{equation:associative-too}--\eqref{equation:great-too}
in equational logic iff it is satisfied by  
the algebra  $(\I, \iii,\neg,\Delta,\too)$.
\end{proposition}

\begin{proof} A routine transcription.
\end{proof}

 \section{The interval functor for general classes of ordered algebras}
 \label{section:general}
Using MV-algebras and IMV-algebras
as a template,  in this section we will  extend the
definition of ``interval algebra'' $\mathcal I(A)$ to  
all algebras $A$ in very general  classes
$\class C$ of ordered structures. 
We will introduce  the associated  class
$\mathcal I (\class C)$
 of interval algebras, and provide a  necessary and 
 sufficient condition for 
 $\mathcal I $ to be a categorical equivalence
 between  $\class C$ and $\mathcal I (\class C)$.
 
In our  approach to IMV-algebras
   the underlying  order of MV-algebras 
has been overshadowed by  the
 monoidal operations $\oplus$ and $\odot$.
 By contrast, 
 the partial order of any algebra $A$ considered in this
 section will play a fundamental role from the outset  in defining
 its associated interval algebra: indeed, any other 
 operation of $A$ will be required to be monotone or
 antimonotone in each input argument.
As another dissimilarity from  IMV-algebras, the 
operations on the intervals of $A$ are no longer  defined
in terms of Minkowski pointwise operations---because
the counterpart of Proposition \ref{proposition:riesz}
need no longer hold for $A$ (see Example~\ref{Ex:Heyting}).

A suitable framework for our generalized approach to
interval algebras is provided by quasivarieties of
``$\rho$-partially ordered algebras'', where 
$\rho $ is a polarity as defined by Pigozzi
in  \cite{DPig}. 

As usual, for any  set $\Sigma$ of constant and functions
 symbols, and  $\Sigma$-algebra $A$, 
 we write  $c^A$ and $f^A$ for  the interpretation in $A$ of
 the constant symbol $c$  (resp., the function symbol $f$) of $\Sigma$.
 More generally, for each $\Sigma$-term
 $t$, we let  $t^A$ denote the interpretation of $t$ in  $A$. 

\begin{definition}
[{\bf \cite[Definition~2.1]{DPig}}]
\label{definition:polarity}
 A {\it polarity} for an algebraic language $\Sigma$ is a map $\rho$ 
 defined on each symbol in $\Sigma$
 such that whenever  $f\in\Sigma$ is an $n$-place
  function symbol, $\rho(f)\in\{+,-\}^{n}$. 
Given a $\Sigma$-algebra $A$ and a partial order 
$\leq$ on $A$, we say that $\leq$ is a {\it $\rho$-partial order of $A$}
 if for each $f\in\Sigma$ of arity $n$,  
 $(a_1,\ldots,a_n)\in A^n$,  $k\leq n$,  and  $a, b\in A$ with $a\leq b$,
  the  following conditions hold:

\medskip
\noindent
$(+)$\,\,\,\,\, If \,\,\,$\rho(f)_k=+$\,\,\,\, then
 $$f^{A}(a_1,\ldots,a_{k-1},a,a_{k+1},\ldots,a_n)\leq f^{A}(a_1,\ldots,a_{k-1},b,a_{k+1},\ldots,a_n);$$

\medskip
\noindent
$(-)$\,\,\,\,\,
If \,\,\,$\rho(f)_k=-$\,\,\,\, then 
$$f^{A}(a_1,\ldots,a_{k-1},b,a_{k+1},\ldots,a_n)\leq f^{A}(a_1,\ldots,a_{k-1},a,a_{k+1},\ldots,a_n).$$

\noindent
The pair $(A,\leq)$ is called a 
{\it $\rho$-partially ordered $\Sigma$-algebra}, 
(for short, {\it  $\rho$-poalgebra} when $\Sigma$ is clear from the context).
We further say that  $(A,\leq)$ is
a  {\it bounded}
 $\rho$-poalgebra if there exist
  constant symbols $0,1\in\Sigma$ 
  such that $0^{A}\leq a\leq 1^{A}$ for each $a\in A$.
\end{definition}

\begin{definition}
Given a polarity  $\rho$   for an algebraic language $\Sigma$ and   a bounded $\rho$-poalgebra
$(A,\leq)$
 we let
 $\mathcal I (A)$
   be the $(\Sigma\cup\{\Delta,\nabla, \iii\})$-algebra whose universe is the set $I(A)
   =\{[a,b]\mid a,b\in A\mbox{ and }a\leq b\}$ of intervals in $A$, 
and whose operations are defined as follows:

\medskip

\begin{itemize}
\item[(i)] For each $n$-ary symbol $f\in\Sigma$ and intervals
 $[a_1,b_1],\dots,[a_n,b_n]\in I(A)$,
\[f^{\mathcal{I}(A)}([a_1,b_1],\dots,[a_n,b_n])
=[f^A(c_1,\ldots,c_n),f^A(d_1,\ldots,d_n)],\]
where $c_k=a_k$ and $d_k=b_k$ if $\rho(f)_k=+$, 
and  $c_k=b_k$ and $d_k=a_k$ if $\rho(f)_k=-$.

\medskip
\item[(ii)]  For each $[a,b]\in I(A)$, \,\,\,
$\Delta^{\mathcal{I}(A)}([a,b])=[a,a]\,\,\, 
\mbox{ and }\,\,\,\nabla^{\mathcal{I}(A)}([a,b])=[b,b]$.

\medskip
\item[(iii)] 
 $\iii^{\mathcal{I}(A)}=[0^{A},1^{A}].$
\end{itemize}
\end{definition}

\begin{remark}
By Proposition \ref{proposition:riesz},  
for any MV-algebra
 $A$  (equipped with its natural  order
$\leq$),   the algebra  $\mathcal{I}(A)$
given by Definition \ref{definition:polarity} 
coincides with the algebra defined in Section~\ref{Sec:IMV}
in terms of Minkowski operations.
It turns out that  Proposition
  \ref{proposition:riesz} 
cannot be extended to arbitrary classes of $\rho$-poalgebras.
For each $\rho$-poalgebra $A$, $n$-ary function symbol $f\in \Sigma$, 
and  intervals $[a_1,b_1],\dots,[a_n,b_n]\in I(A)$ we have the inclusion
\[
\{f^A(e_1,\ldots,e_n)\mid(e_1,\ldots,e_n)\in[a_1,b_1]\times\cdots\times[a_n,b_n] \}
\subseteq f^{\mathcal{I}(A)}([a_1,b_1],\dots,[a_n,b_n]).
\]
The converse inclusion need not hold
in general, because  the set 
\[f^A([a_1,b_1]\times\cdots\times[a_n,b_n])
=\{ f^A(e_1,\ldots,e_n)\mid(e_1,\ldots,e_n)\in[a_1,b_1]\times\cdots\times[a_n,b_n] \}
\] 
need not  be an interval. 
However, $f^{\mathcal I(A)}([a_1,b_1]\times\cdots\times[a_n,b_n])$ has a 
smallest and a largest  element. 
Thus
$f^{\mathcal{I}(A)}([a_1,b_1],\dots,[a_n,b_n])$ coincides with the smallest
interval of $A$ containing the set 
$f([a_1,b_1]\times\cdots\times[a_n,b_n])$.
\end{remark}

\begin{example}\label{Ex:Heyting}
Following \cite[p.~44]{bursan}, by a {\it Heyting algebra} 
$(A, \wedge,\vee,\to,0,1)$ we mean a
bounded distributive lattice such that 
$a\wedge b \leq c$ iff $a\leq b\to c$.
The underlying order  $\leq$ of $A$ is  given by
the stipulation:   $a\leq b$ iff $a\wedge b=a$. 
As is well known, the class
$\class H$   of Heyting algebras is a variety. 
Let 
 us denote by  $[0,1]$  the uniquely determined
 Heyting algebra whose lattice reduct is 
 $([0,1], \min,\max,0,1)$. 
 Then  $a\to b=1$ if $a\leq b$,  and $a\to b=b$ otherwise. 
For each $[a,b],[c,d]\in\mathcal{I}([0,1])$ we can write
\begin{align*}
[a,b]\wedge^{\mathcal{I}(A)}[c,d]&=[\min\{a, c\},\min\{b,d\}]=\min([a,b]\times[c,d]);\\
[a,b]\vee^{\mathcal{I}(A)}[c,d]&=[\max\{a, c\},\max\{b,d\}]=\max([a,b]\times[c,d]).
\end{align*} 
However, when  $a\neq b$ we have
$
[a,b]\to^{\mathcal{I}(A)}[a,a]=[b\to a, a\to a]=[a,1], 
$
but
$\{e\to f\mid (e,f)\in[a,b]\times[a,a]\}=\{a,1\}$.
\end{example}

\begin{definition}
\label{definition:order}
Given a class of $\class C$ of $\rho$-poalgebras,
  a set of equations $E(x,y)$  in two variables is said to
   {\it determine the order} of $A$ if the following two conditions
   are equivalent, for all  
   $a,b\in A$:
   \begin{itemize}
   \item $a\leq b$;
   \item $t^{A}(a,b)=s^{A}(a,b)$ for each equation $t(x,y)=s(x,y)\in E(x,y)$.
   \end{itemize} 
\end{definition}

As an example,  for every MV-algebra $A$,
the single equation $\neg x\oplus y=1$ determines the order of $A$.
In any class of lattices, the order is determined by the single equation
 $x\wedge y=x$.

\subsection*{Notational convention}
Throughout  the rest of this section,  $\Sigma$ will denote a
 set of constant and function symbols,  and 
 $\class Q$ will denote a $\Sigma$-quasivariety  of $\rho$-poalgebras  having a set    $E(x,y)$ of $\Sigma$-equations
 such that  the order of each algebra $A\in \class Q$ is determined by $E(x,y)$. 
We also write
 $$\class{IQ}=\mathbb{Q}(\mathcal{I}(\class Q))$$
for  the quasivariety of algebras 
generated by $\mathcal{I}(\class Q)$. We  regard
 $\class Q$ and $\class{IQ}$ as categories whose morphisms are $\class Q$-homomorphisms and
$\class{IQ}$-homomorphisms, respectively.

\begin{theorem}\label{Theo:GeneralFunctor1}
For any $\class Q$ we have:  
\begin{itemize}
\item[(i)] Every $\class Q$-homomorphism $h\colon A\to B$ 
is order-preserving.  

\medskip
\item[(ii)] 
Let the assignment    $\mathcal I\colon \class Q\to \class{IQ}$  be defined by: 
\begin{eqnarray*}
\mbox{objects:} \,\,\,\,\,\,\,\,\,\,\,\,\,\,\,\,\,\,\,\,\,\,\,\,\,\, \,\,\,\,\,\,\,\,\,\,\,\,\,\,\,\,\,\,\,\, 
 A  \,\,\,\,&\mapsto&\,\,\,\, \mathcal I( A)\\
\mbox{morphisms:} \,\,\,\,\,\,\,\,\,\,\,\, \,\, h\colon A\to B\,\,\,\,& \mapsto&\,\,\,\,
 \mathcal I(h)\colon \mathcal I(A)\to\mathcal I(B),
\end{eqnarray*}

\medskip
\noindent
where the homomorphism  $\mathcal I(h)$
is given by $(\mathcal I(h))([a,b])=[h(a),h(b)]$, for any interval
$[a,b]\in \mathcal I( A)$.  Then $\mathcal I$
is a well-defined functor.
\end{itemize}

\end{theorem}
\begin{proof}
(i) follows because
 the same equations $E(x,y)$ define  the order in $A$ and in~$B$.

\smallskip
(ii) By (i), \,
 $\mathcal{I}(h)$ is a well defined map from $\mathcal{I}(A)$ to $\mathcal{I}(B)$. To see that $\mathcal{I}(h)$ is a homomorphism,  first consider $f\in \Sigma$ and $[a_1,b_1],\dots,[a_n,b_n]\in \mathcal I(A)$. 
Then
\begin{align*}
\mathcal{I}(h)(f^{\mathcal I(A)}([a_1,b_1],\dots,[a_n,b_n]))&=\mathcal{I}(h)([f^A(c_1,\ldots,c_n),f^A(d_1,\ldots,d_n)])\\
&=[h(f^A(c_1,\ldots,c_n)),h(f^A(d_1,\ldots,d_n))],
\intertext{where $c_k=a_k$ and $d_k=b_k$ if $\rho(f,k)=+$; and  $c_k=b_k$ and $d_k=a_k$ if $\rho(f,k)=-$. Since $h$ commutes with $f$, }
\mathcal{I}(h)(f^{\mathcal I(A)}([a_1,b_1],\dots,[a_n,b_n]))&=[f^{B}(h(c_1),\ldots,h(c_n)),f^{B}(h(d_1),\ldots,h(d_n))]\\
&=f^{\mathcal I(B)}([h(a_1),h(b_1)],\ldots,[h(a_n),h(b_n)]).
\end{align*}
For every   $[a,b]\in\mathcal I(A)$,
\begin{align*}
\mathcal{I}(h)\bigl(\Delta^{\mathcal I(A)}[a,b]\bigr)&=\mathcal{I}(h)([a,a])=[h(a),h(a)]=\Delta^{\mathcal I(B)}[h(a),h(b)]\\
&=\Delta^{\mathcal I(B)}\bigl(\mathcal{I}(h)([a,b])\bigr)\\
\mathcal{I}(h)\bigl(\nabla^{\mathcal I(A)}[a,b]\bigr)&=\mathcal{I}(h)([b,b])=[h(b),h(b)]=\nabla^{\mathcal I(A)}[h(a),h(b)]\\
&=\nabla^{\mathcal I(A)}\bigl(\mathcal{I}(h)([a,b])\bigr)\\
\mathcal{I}(h)\bigl(\iii^{\mathcal{I(A)}}\bigr)&=\mathcal{I}(h)([0^A,1^A])
=[h(0^A),h(1^A)]=[0^B,1^B]=\iii^{\mathcal{I}(B)}.
\end{align*}
Therefore,  $\mathcal{I}$ is a functor.
\end{proof}

For each quasivariety of $\rho$-partially ordered algebras $\class Q$ the functor $\mathcal I\colon \class Q\to \class{IQ}$ defined in Theorem~\ref{Theo:GeneralFunctor1} is called the {\it interval functor} of $\class{Q}$.

\medskip
Next we  will prove that $\mathcal I$ has always a left adjoint, and 
$\class Q$ is isomorphic to a retractive subcategory of $\class{IQ}$.

\begin{theorem}\label{Theo:GeneralFunctors}
The following conditions hold:

\smallskip
\begin{itemize}

\item[(i)] For each $J\in\class{IQ}$ and 
 $f\in \Sigma$,  the set 
$C(J)=\{x\in J\mid \nabla x=x=\Delta x\}$ 
is closed under the operation   $f^{J}$.

\bigskip
\item[(ii)]  Let the assignment \,\, $\mathcal C\colon \class{IQ}\to\class Q$  be
defined by:
\begin{eqnarray*}
\mbox{objects:}  \,\,\,\,\,\,\,\,\,\,\,\,\,\,\,\,\,\,\,\,\,\,\,\,\,\,\,\,\,\,\,\,\,\,\,\,\,\,\,\,\,\,\,\,\,\,\,\,
\,\,\,\,\,\, \,\,\,
J \,\,\,\,&\mapsto& \mathcal C( J)\\
\mbox{morphisms:}\,\,\,\,\,\,\,\,\,\,\,\, \,\,\,\,\,\,\,\,\,\,\,
\rho\colon J\to K\,\,\,\,\,\,\, &\mapsto& \ \mathcal C( \rho)=
\rho\restrict{C(J)}, 
\end{eqnarray*} 
where $\mathcal C(J)$ is the $\Sigma$-algebra whose universe is $C(J)$ and for each $f\in \Sigma$, $f^{\mathcal C(J)}=f^J\restrict {C(J)}$.
Then $\mathcal C$
is a well-defined faithful functor.

\medskip

\item[(iii)] For each $A\in\class Q$ the map $\iota_A\colon A\to \mathcal C(\mathcal I(A))$ defined by $\iota_A(a)=[a,a]$ is an isomorphism; further,  the map
$A\mapsto \iota_A$ 
is a natural isomorphism from the identity 
functor ${\rm I}_{\class Q}\colon \class Q\to \class Q$ 
to the composite functor $\mathcal{C}\circ \mathcal{I}$.

\medskip

\item[(iv)] For each $J\in\mathbb{Q}(\mathcal{I}(\class Q))$, 
 the map $\gamma_{J}\colon J\to \mathcal{I}(\mathcal{C}(J))$ defined by $\gamma_{J}(a)=[\Delta a,\nabla a]$ is a one-to-one homomorphism; the map 
$J\mapsto \gamma_A$ 
is a natural transformation from the identity 
functor ${\rm I}_{\class{IQ}}\colon\class{IQ}\to\class{IQ}$ 
to the composite functor $\mathcal{I}\circ \mathcal{C}$.

\end{itemize}

\end{theorem}
\begin{proof}

(i) 
Let $A\in\class Q$
and $[a,b]\in \mathcal{I}(A)$. Observe that $\Delta [a,b]=[a,b]$ implies $b=a$. 
Therefore, for any $n$-ary function symbol
 $f\in \Sigma$ and $[a_1,b_1],\ldots,[a_n,b_n]\in \mathcal{I}(A)$,
 if   $\Delta[a_i,b_i]=[a_i,b_i]$  for each\, $i$,   then 
\begin{align*}
f^{\mathcal{I}(A)}([a_1,b_1],\ldots,[a_n,b_n])&=f^{\mathcal{I}(A)}([a_1,a_1],\ldots,[a_n,a_n])\\
&=[f(a_1,\ldots, a_n),f(a_1,\ldots, a_n)]\\
&=\Delta^{\mathcal{I}(A)}\bigl( f^{\mathcal{I}(A)}([a_1,b_1],\ldots,[a_n,b_n])\bigr).
\end{align*}
As a consequence, the quasiequation 
\begin{equation}\label{Eq:QuasiOper}\Delta x_1=x_1,\ldots,\Delta x_n=x_n\Rightarrow \Delta f(x_1,\ldots, x_n)=f(x_1,\ldots,x_n)
\end{equation}
 is satisfied by  each algebra $\mathcal{I}(\class Q)$.
Since each $J\in\class{IQ}$ satisfies \eqref{Eq:QuasiOper},  $C(J)$ is 
closed under  $f^{J}$.

\medskip
(ii) Let $s_1=t_1,\ldots, s_n=t_n \Rightarrow s=t$ be a quasiequation
in the variables $x_1,\ldots, x_m$ satisfied by all algebras in  
  $\class Q$ . The same argument used in (i) shows that
  the quasiequation 
\begin{equation}\label{eq:IQuasi}
\Delta x_1=x_1,\ldots, \Delta x_m=x_m,s_1=t_1,\ldots, s_n=t_n\Rightarrow s=t
\end{equation}
is satisfied by  every algebra in $\mathcal{I}(\class Q)$. Since  $\mathcal C(J)$ 
satisfies every quasiequation satisfied by  $\class Q$, then 
  $\mathcal C(J)\in \class Q$. Every homomorphism $h\colon J\to K$
   commutes with $f^{J}$ (for each $f\in\Sigma$), whence  
  $\mathcal C(h)$ is a homomorphism.

To prove that $\mathcal C$ is faithful, let $J,K\in\class{IQ}$ and 
$g,h\colon J\to K$  be  homomorphisms such that 
$\mathcal{C}(g)=\mathcal{C}(h)$. 
The quasiequation 
\begin{equation}\label{eq:Moisil}
\Delta x=\Delta y,\nabla x=\nabla y\Rightarrow x=y
\end{equation} 
is satisfied by  each algebra in $\mathcal{I}(\class Q)$,  whence it is
satisfied by  $J$.
Now  for each $a\in J$
we  can write 
$\Delta g(a)=g(\Delta a)=h(\Delta a)=\Delta h(a)$ 
and $\nabla g(a)=g(\nabla a)=h(\nabla a)=\nabla h(a)$. 
Since $K$ satisfies \eqref{eq:Moisil}, is follows that $g(a)=h(a)$ for each $a\in K$, that is, $g=h$.

\medskip
(iii) It is easy to check that the map $\iota_{A}\colon A\to \mathcal{C}(\mathcal{I}(A))$ is an isomorphism, which  is natural in $\class Q$.

\medskip
(iv) 
To see that $\gamma_{J}$  is well defined observe that 
\begin{equation}\label{eq:order}
t(\Delta x,\nabla x)=s(\Delta x,\nabla x) 
\end{equation}
is satisfied by all algebras of 
 $\mathcal{I}(\class Q)$, for each $s=t\in E(x,y)$. Then for each $a\in J$, $\Delta a \leq \nabla a$ in $\mathcal{C}(J)$.
The fact that $\gamma_{J}$ is one-to-one follows directly from  \eqref{eq:Moisil}. 
For each 
$A\in\class Q$ and  $f\in \Sigma$, let 
$$
\mbox{$t_{f,i}(x_i)=\Delta x_i$ 
and $s_{f,i}(x_i)=\nabla x_i$ if $\rho(f)_i=+$,}
$$
$$
\mbox{ 
$t_{f,i}(x_i)=\nabla x_i$ and 
$s_{f,i}(x_i)=\Delta x_i$  if $\rho(f)_i=-$.}
$$

\smallskip
\noindent
By definition of 
$f^{\mathcal{I}(A)}$, 
the equations
\begin{align}
\label{eq:DeltaF}\Delta(f(x_1,\ldots, x_n))&=f(t_{f,1}(x_1),\ldots, t_{f,n}(x_n))\\
\label{eq:NablaF}\nabla(f(x_1,\ldots, x_n))&=f(s_{f,1}(x_1),\ldots, s_{f,n}(x_n))
\end{align}
are satisfied by every  algebra in $\mathcal{I}(\class Q)$.
Since $J$ satisfies these equations, 
$\gamma_{J}$ preserves $f^{J}$ for each $f\in\Sigma$.
Finally,  a direct inspection shows that every algebra in
$\mathcal{I}(\class Q)$ satisfies the following equations: 
\begin{equation}
\label{eq:extraoper}
\Delta \iii=0, \,\,\, \nabla \iii=1, \,\,\,
\Delta \Delta x=\Delta x, 
\,\,\,
\nabla \Delta x=\Delta x, 
\,\,\,
\nabla \nabla x=\nabla x,
\,\,\,
 \Delta\nabla x=\nabla x. 
\end{equation}
We have shown
that    $\gamma_{J}$ also preserves $\iii^{J}$, $\Delta^J$ and $\nabla^J$,
as required to complete the proof.
\end{proof}

 \medskip
The proof of Theorem~\ref{Theo:GeneralFunctors}
yields  a method which,
having in input a  set  of quasiequations axiomatizing $\class Q$ 
outputs  a  set of quasiequations axiomatizing  $\class{IQ}$:

\begin{corollary}
Suppose $\class Q$ is axiomatized by a finite set
 $M$ of $\Sigma$-quasiequations.
 Then  $\class{IQ}$ is axiomatized by the following
 quasiequations,
  for every  function symbol $f\in\Sigma$, \,\,\,
quasiequation $s_1=t_1,\ldots, s_n=t_n\Rightarrow s=t\in M$\,\,\,
and equation $s=t\in E(x,y)$:
\begin{align}
\Delta x_1=x_1,\ldots,\Delta x_n=x_n&\Rightarrow \Delta f(x_1,\ldots, x_n)=f(x_1,\ldots,x_n)\tag{\ref{Eq:QuasiOper}}\\
\Delta x_1=x_1,\ldots, \Delta x_m=x_m,s_1=t_1,\ldots, s_n&=t_n\Rightarrow s=t\tag{\ref{eq:IQuasi}}\\
\Delta x=\Delta y,\,\,\,\nabla x=\nabla y\,\,&\Rightarrow\, x=y\tag{\ref{eq:Moisil}}
\end{align}

\medskip
\noindent
together with the equations 
\begin{align}
\tag{\ref{eq:order}}t(\Delta x,\nabla x)&=s(\Delta x,\nabla x) \\
\quad\quad\quad\quad\quad\tag{\ref{eq:DeltaF}}\Delta(f(x_1,\ldots, x_n))&=f(t_{f,1}(x_1),\ldots, t_{f,n}(x_n))\\ 
\tag{\ref{eq:NablaF}}\nabla(f(x_1,\ldots, x_n))&=f(s_{f,1}(x_1),\ldots, s_{f,n}(x_n)),
\\
\tag{\ref{eq:extraoper}}\Delta \iii=0, \,\,\nabla \iii=1, \,\,\Delta \Delta x=\Delta x,
 \,\,\nabla \Delta x&=\Delta x,\,\,\nabla \nabla x=\nabla x, \,\, \Delta\nabla x=\nabla x. 
\end{align}
\end{corollary}

\begin{proof}
From the proof of Theorem~\ref{Theo:GeneralFunctors}
it follows that  each algebra in $\class{IQ} $ satisfies 
\eqref{Eq:QuasiOper}-\eqref{eq:extraoper}.
Conversely, let  $J$ be an algebra satisfying
 \eqref{Eq:QuasiOper}-\eqref{eq:extraoper}.
Mimicking the proof of  Theorem~\ref{Theo:GeneralFunctors}(i), 
from  \eqref{Eq:QuasiOper} it follows that the set
 $C(J)=\{a\in J\mid \Delta a=a=\nabla a\}$ is 
 closed under $f^{J}$ for each $f\in \Sigma$.
By  $\eqref{eq:IQuasi}$,  the algebra \, $\mathcal C(J)$ 
equipped with all restrictions  $f^{J}\restrict {C(J)}$
 belongs to $\class Q$.
Further, from
 \eqref{eq:order} 
 we get  $\Delta a\leq \nabla a$ in $\mathcal C({J})$, and 
 hence the map $\gamma_{J}(a)=
 [\Delta a, \nabla a]\colon J\to \mathcal{I}(\mathcal C(J))$
  is well defined.
By  \eqref{eq:Moisil}, $\gamma_J$ is one-to-one;  
by \eqref{eq:DeltaF}-\eqref{eq:extraoper}, $\gamma_J$ is a homomorphism. 
In conclusion,   $J\in\mathbb{IS}(\mathcal{I}(\class Q))
= \class{IQ}$.
\end{proof}

 The following
 theorem  gives 
  necessary and sufficient conditions on a quasivariety $\class{Q}$ for its interval functor 
   $\mathcal I$ to be a categorical equivalence.

\begin{theorem}\label{Theo:EquivConditions}
The following conditions  are equivalent:

\begin{itemize}
\item[(i)] The interval  functor $\mathcal{I}\colon\class Q\to \class{IQ}$ is
 a categorical equivalence.
 
 \medskip
\item[(ii)] $\class{IQ}=\mathbb{I}(\mathcal{I}(\class{Q}))= $ the class of isomorphic
copies of algebras of $\mathcal{I}(\class{Q})$.

 \medskip
\item[(iii)]  For
 each $A\in \class Q$ the set $C(\mathcal{I}(A))
 =\{[a,a]\mid a\in A\}$ generates $\mathcal{I}(A)$.

 \medskip
\item[(iv)]  For some $(\Sigma\cup\{\Delta,\nabla,\iii\})$-term 
$t(y,z)$,  the equation
$t(\Delta(x),\nabla(x) )=x$ is satisfied by  every algebra of 
$\class{IQ}$.
\end{itemize}

\smallskip
\noindent
In particular, if  $\class Q$ is a variety,  these conditions
 (i)-(iv)   are also  equivalent to

\smallskip
\begin{itemize}
\item[(v)] $\mathbb{V}(\mathcal{I}(\class{Q}))=\mathbb{I}(\mathcal{I}(\class{Q}))$, where $\mathbb{V}(\mathcal{I}(\class{Q}))$ is the variety generated by $\mathcal{I}(\class{Q})$.
\end{itemize}
\end{theorem}

\begin{proof}
(i)$\Rightarrow$(ii). Trivially, if $\mathcal{I}$ is
 an equivalence every $J\in \class{IQ}$ belongs to $\mathbb{I}(\mathcal{I}(\class Q))$.

\smallskip
(ii)$\Rightarrow$(iii). Let $A\in \class{Q}$ and $J\subseteq \mathcal{I}(A)$ be the subalgebra generated by $\{[a,a]\mid a\in A\}$. Then $C(J)=\{[a,a]\mid a\in A\}$ and $\mathcal{C}(J)\cong A$.
Let $g\colon J\to  \mathcal{I}(A)$ be the inclusion map.
 By (ii), there exists $B\in\class Q$, and an isomorphism 
$f\colon J\to \mathcal{I}(B)$. 
By Theorem~\ref{Theo:GeneralFunctors}(iii), 
$B\cong\mathcal{C}(\mathcal I(B))=f(\mathcal{C}(J))\cong \mathcal{C}(J)\cong A$. 
The map $h\colon A \to B$ defined by $h(a)=b$ whenever 
$f([a,a])=[b,b]$ is an isomorphism. 
Therefore,   $k=f\circ \mathcal{I}(h^{-1})\colon J\to \mathcal{I}(A)$ is 
an isomorphism onto  $\mathcal{I}(A)$, and  
\[k([a,a])=f\circ \mathcal{I}(h^{-1})([a,a])=f([h^{-1}(a),h^{-1}(a)])=[a,a].\] 
Having thus proved  $k=g$, we conclude that $J=g(J)=k(J)=\mathcal{I}(A)$.

\smallskip
(iii)$\Rightarrow$(iv). 
Let $\F(x)\in \class{IQ}$ be the free algebra with one free generator. 
By Theorem~\ref{Theo:GeneralFunctors}(iv), $\F(x)$ is isomorphic to a 
subalgebra of $\mathcal{I}(\mathcal C(\F(x)))$. 
By (iii), $\mathcal{I}(\mathcal C(\F(x)))$ is generated by 
$C(\mathcal {I}(\mathcal C(\F(x))))$,  $C(\F(x)) $ generates $\F(x)$ 
and there exists a  $(\Sigma\cup\{\Delta,\nabla,i\})$-term
 $s(x_1,\ldots, x_n)$  together with  elements $a_1,\ldots, a_n$ in $C(\F(x))$ 
such that $s^{\F(x)}(a_1,\ldots, a_n)=x$.
Since $\F(x)$ is generated by $x$, $\mathcal C(\F(x))$ is generated 
by $\Delta x$ and $\nabla x$. Thus for each $i\in\{1,\ldots,n\}$ 
there is a term $t_i(y,z)$ such that $a_i=t_i^{\F(x)}(\Delta x,\nabla x)$. 
Therefore, the term $t(y,z)=s(t_1(y,z),\ldots, t_n(y,z))$ 
satisfies
$t^{\F(x)}(\Delta x,\nabla x)=x$. 
Since $\F(x)$ is the free algebra
 in 
$\mathbb{Q}(\mathcal{I}(\class Q))$
 with free generator $x$,
the equation $t(\Delta x,\nabla x)=x$ is 
satisfied by  every algebra of
$\mathbb{Q}(\mathcal{I}(\class Q))$.

\smallskip
(iv)$\Rightarrow$(i). 
Since, by Theorem~\ref{Theo:GeneralFunctors}(iii)  $\iota_{A}$ is an isomorphism for each $A\in \class Q$, it is enough to check 
 that $\gamma_{J}$ is an isomorphism for each $J\in \class{IQ}$. By Theorem~\ref{Theo:GeneralFunctors}(iv), $\gamma_{J}$ is one-to-one. 
 Letting $[a,b]\in \mathcal{I}(\mathcal{C}(J))$ we can write 
\begin{align*}
\gamma_{J}(t^{J}(a,b))&=t^{\mathcal{I}(\mathcal{C}(J))}(\gamma_{J}(a),\gamma_{J}b)
=t^{\mathcal{I}(\mathcal{C}(J))}([\Delta^J a,\nabla^J a],[\Delta^J b,\nabla^J b])\\
&=t^{\mathcal{I}(\mathcal{C}(J))}([ a,a],[ b, b])
=t^{\mathcal{I}(\mathcal{C}(J))}(\Delta^{\mathcal{I}(\mathcal{C}(J))}[ a,b],\nabla^{\mathcal{I}(\mathcal{C}(J))}[ a, b])\\
&=[a,b],
\end{align*}
and  $\gamma_J$ is onto $\mathcal{I}(\mathcal{C}(J))$.

\smallskip
Trivially (v) implies (ii) even if  $\class Q$ is not a variety. 
Now  assume $\class Q$ is a variety and  (i)-(iv) holds. 
Let $K\in \mathbb{V}(\mathcal{I}(\class Q))$. 
There exists $J\in \class{IQ}$ and a homomorphism $h\colon J\to K$ onto $K$.
Since  $C(K)=\{x\in K\mid \Delta x=\nabla x=x\}=h(C(J))$
then  $C(K)$ is closed under $f^{K}$ for each $f\in \Sigma$. 
Let $\mathcal{C}(K)$ be the $\Sigma$-algebra whose universe is 
$C(K)$ and whose operations  are given by 
restricting to  $C(K)$ the operations of $K$. 
 Thus the map $h\restrict{C(J)}\colon \mathcal{C}(J)\to \mathcal{C}(K)$ is a 
homomorphism onto $\mathcal{C}(K)$. 
Since $\class Q$ is a variety, $\mathcal{C}(K)\in \class Q$. 
Since $J$ satisfies    \eqref{eq:DeltaF}--\eqref{eq:NablaF}
then so does $K$. As a consequence, the map 
$\gamma_{K}\colon K\to \mathcal{I}(\mathcal{C}(K))$
defined by $\gamma_{K}(a)=[\Delta a,\nabla a]$ is a 
homomorphism. 
Finally, let 
  $a,b\in K$ be such that 
$\gamma_{K}(a)=\gamma_{K}(b)$. Equivalently,  $\Delta a=\Delta b$ and 
$\nabla a=\nabla b$. 
 Since $K\in \mathbb{V}(\mathcal{I}(A))$, recalling (iv)
we can write 
\[a=t^{K}(\Delta^K a,\nabla^K a)=t^{K}(\Delta^K b,\nabla^K b)=b,\] 
which shows that
$K\in \mathbb{IS}(\mathcal{I}(\class Q))$. 
By (iii),
  $\class{IQ}=\mathbb{IS}(\mathcal{I}(\class Q))
  =\mathbb{I}(\mathcal{I}(\class Q))$. 
  We have proved that condition (v) follows from
   (i)-(iv), as desired.
\end{proof}

\begin{corollary}
Let  $\class R\subseteq \class Q$ be a subquasivariety of $\class Q$ and $\class{IR}=\mathbb{Q}(\mathcal{I}(\class R))$. If the functor $\mathcal{I}\colon\class Q\to \class{IQ}$ 
is a categorical equivalence, then  its restriction
 $\class R$ is a categorical equivalence between  $\class R$ and 
  $\class{IR}$.
\end{corollary}
\begin{corollary}
Every quasivariety of MV-algebras is categorically equivalent to the quasivariety of its interval algebras.
\end{corollary}

The next two corollaries show that $\mathcal{I}$ is a categorical equivalence for several classes of lattice ordered algebras having an important role in (many-valued) logic.

\begin{corollary}
\label{corollary:extensive}
Suppose the quasivariety  $\class Q$ {\em has a  lattice reduct.}
In other words,  there are  binary operation symbols $\wedge ,\vee \in\Sigma$ such that 
for each algebra $A\in \class Q$ the $(\wedge,\vee)$-reduct of $A$ is a lattice.
Then   the interval functor $\mathcal I\colon \class Q\to \class{IQ}$ is an equivalence.
\end{corollary}
\begin{proof}
Directly from  Theorem~\ref{Theo:EquivConditions}, using the $\class{IQ}$-term $t(y,z)=(\iii \vee y)\wedge z$.
\end{proof}

\begin{corollary}
\label{corollary:listone}
 For each subquasivariety of the
following varieties,  the  functor $\mathcal I$ is an equivalence:
\begin{itemize}
 \item[---] Heyting algebras \cite[p.~44]{bursan},  BL-algebras \cite{[21]},
  MTL--algebras \cite{GalJipKowOno},  and, more generally,  residuated lattices \cite{GalJipKowOno}.
 \item[---] Modal algebras \cite{ChaZak}.
 \end{itemize}
\end{corollary}
 
\begin{remark}
\label{remark:salvagente}
Corollary  \ref{corollary:extensive}
does not apply directly to MV-algebras, 
since the lattice structure of an  MV-algebra is not 
a reduct of the MV-structure, but, rather, it is  term-definable from the
basic operations  $\neg,\oplus,\odot$. 
However, by \eqref{equation:great}, the term $t(y,z)=y\oplus(\iii\odot z\odot\neg y)$
does satisfy condition (iv) of Theorem~\ref{Theo:EquivConditions}. 
In this way,  Theorem \ref{theorem:equivalence} can be seen as a consequence of 
Theorem~\ref{Theo:EquivConditions}(iv)$\to$(i).
The IMV-term
 $t(y,z)=(\iii \boldsymbol{\vee} y)\boldsymbol{\wedge} z$, 
(where $ \boldsymbol{\vee}$ and $\boldsymbol{\wedge} $
are as defined in 
\eqref{equation:central-order})
 satisfies condition (iv) of Theorem~\ref{Theo:EquivConditions}.
This is a  very special feature of MV-algebras, depending
on the actual definition of  $\boldsymbol{\vee}, \boldsymbol{\wedge}$,
as well as on the identity
 $\neg \iii=\iii$ being satisfied by all IMV-algebras. 
\end{remark}

 To find an example of a quasivariety
  for which $\mathcal{I}$ does not determine a categorical equivalence,
   we need to leave the domain of lattice ordered structures. 
The following example exhibits a  class of algebras whose underlying
 order is determined by the equation 
 $a\to b=1$,  but where the interval functor 
  $\mathcal{I}$ is not a categorical equivalence:

\begin{example}
Following \cite{Diego},   a  {\it Hilbert algebra} is an algebra $(A,\to,1)$ satisfying the equations $a\to(b\to a)=1$, \, $(a\to (b\to c))\to((a\to b)\to (a\to c))=1$ and the quasiequation $a\to b=1,b\to a \Rightarrow a=b$.
   Hilbert algebras are  the algebraic equivalent semantics
 of  the implicational fragment of intuitionistic logic. A {\it bounded}
  Hilbert algebra is a structure $(A,\to,0,1)$  such that $(A,\to,1)$ is a Hilbert algebra and   $0\to a=1$ for each $a\in A$. 
  
  Let $\class{BH}$ denote the variety of bounded Hilbert algebras. 

Every Hilbert algebra $A$ admits a natural order
 defined by $a\leq b$ if $a\to b=1$. With this order, $A$  satisfies $0\leq a\leq 1$. Since $\to$ reverses the order in the first coordinate and preserves the order in the second coordinate, upon 
 defining $\rho(\to)=(-,+)$, each bounded  Hilbert algebra
 becomes  a bounded $\rho$-poalgebra.
Let 
$\,\,\,B=(\{0,\,a,\,1\},\,\to~,\,0,\,1)$ be the unique $3$-element bounded Hilbert algebra.
 It is easy to prove that 
$D=\{[0,1],[0,0],[a,a],[1,1],[a,1]\}$ is the universe of  a  subalgebra of $\mathcal{I}(B)$.
By Theorem~\ref{Theo:EquivConditions}(iii),  the functor
 $\mathcal{I}\colon\class{BH}\to \class{IBH}$ is not a categorical equivalence.
\end{example}

\section{Related work}
\label{section:literature}

Corollary \ref{corollary:extensive} shows that
the interval functor of most
quasivarieties  $\class{Q}$  
of partially ordered algebras existing in the literature
 is in fact a categorical equivalence. Thus, intuitively,
 $\class{Q}$  and the quasivariety of its interval algebras
 $\mathcal I(\class{Q})$ stand in the same relation as
  MV-algebras
and IMV-algebras:  the functor
$\mathcal I$ preserves subalgebras, homomorphic images,
products, coproducts, projectives, injectives.

Remarkably enough, as the following brief
survey will show,
  the pervasiveness of  this categorical equivalence 
has gone virtually unnoticed in the vast literature on
interval algebras, triangle algebras, interval constructors
and  triangularizations of algebras whose underlying
order is a lattice.

\subsection*{Interval analysis, Minkowski sums}
From the very outset,
the basic operations in  {\it  interval analysis} include 
Minkowski sum,  \cite[(2.15)]{interval}.  
By contrast, in the literature on interval algebras,
interval constructors, triangularizations, interval t-norms,
  \cite{information, 35years, cina,  gehwal,   van,  van-bis,   web},    
every (monotone) binary operation   $\star$ on the set of intervals 
of a partially ordered algebra  $A$, possibly equipped with  lattice
and/or   t-norm operations,  is 
usually defined by
$[\alpha,\beta]\star [\gamma,\delta]=[\alpha\star\gamma,
\beta\star\delta],$  as we have done
in Section~\ref{section:general}.  This is so because the set
$\{\xi\star\chi\mid \xi\in[\alpha,\beta] ,\chi\in[\gamma,\delta] \}$
need not be an interval of $A$---whenever the counterpart of Proposition
\ref{proposition:riesz} does not hold  for 
$A$.
Proposition \ref{proposition:riesz} holds for any
MV-algebra $A$ because by  \cite[\S 3]{mun86},
 (up to isomorphism),
$A$
 is the unit interval of a unique 
  unital $\ell$-group $(G,u)$,  and  $G$, like the ordered
group of real numbers considered in interval analysis, 
 has the Riesz decomposition
property, \cite[Lemma~1, page 310, and Theorem~49, p. 328]{bir}.

Other basic operations in interval analysis
include  $\underline x,\overline x$,
which are also  found in interval algebra theory, and
correspond to our   $\Delta x,\nabla x$,
to   $\nu x, \mu x$ of 
 \cite{van-bis}, to  $l$, $r$, or $\pi_1,\pi_2$
 of  \cite{information}.
An important property of any
 IMV-algebra $J$ is Proposition \ref{proposition:trivial}(iii),
 stating that  
 \begin{quote} 
 {\small by {\it equation}  \eqref{equation:great}, any
$x\in J$  is uniquely determined 
by $\Delta x$ and $\nabla x$.}
\end{quote}
Mutatis mutandis, for certain classes $\class{K}$ of algebras,
(e.g., the ``triangle algebras'' of \cite[Definition~3]{van-bis}),
the property above turns out to be
definable by equations, thus allowing the equational
definability of the interval algebras of $\class{K}.$
For every quasivariety  $\class{Q}$ considered
in  Corollary  \ref{corollary:extensive}, the
existence of a term $t(\Delta x,\nabla x)$ 
equal to  $x$ yields a categorical equivalence between
$\class{Q}$ and  its
associated category $\mathcal I(\class{Q})$ of interval algebras.

\subsection*{Monadic, modal, Girard algebras}
The operations  $\Delta,\nabla$,
sometimes written as
 $\exists, \forall$, or 
$\Diamond,\Box$,
  also occur in the realm of
{\it monadic algebras}, to express some
properties of quantifiers.
Thus for instance, the paper \cite{dingri} represents  
every monadic MV-algebra
as an algebra of suitable  (generally non-monotone) pairs of MV-algebras,
equipped with unary operators that are reminiscent of our
$\Delta$ and $\nabla$.  However,  
monadic MV-algebras are not (term-equivalent to) IMV-algebras.
A particular case of monadic MV-algebras is given by
monadic boolean algebras: a celebrated theorem by A.Monteiro
\cite{mon} shows that   each
  monadic boolean algebra $A$ yields  a three-valued \luk\ algebra
  (i.e., an MV$_3$-algebra, 
  as defined by Grigolia,
 see \cite[\S 8.5]{cigdotmun} and \cite{cig})
$L_A$,
in such a way that, up to isomorphism,  each 
three-valued \luk\ algebra
arises has the form $L_A$
for some monadic boolean algebra $A$.
And again,  for any boolean algebra $B$,
$\mathcal I(B)$   will not be
an MV$_3$ algebra, because of Proposition 
\ref{proposition:trivial}(iv).

The operations
   $\Delta,\nabla$   are also found in the theory of {\it modal algebras}
    to express
 the notions of ``necessity'' and ``possibility''.
 In the specific domain of MV-algebras, 
modal operators  are studied in 
 \cite{modal-mv}  as a special case
of general modal operators on residuated lattices
 \cite{bou}.
However, the IMV-operations
$\Delta,\nabla$  
  do not satisfy the same
equations of the   modal operators 
of  \cite{modal-mv} and  \cite{bou}. 
Further examples  are mentioned   in
\cite[Section~3.3]{van-bis}.

In the paper  \cite{web},   the set
of intervals in  an MV-algebra  $A$ 
is equipped with the structure
of a ``Girard algebra''  $\mathcal G(A)$. 
The main aim of the paper is
to represent
 conditionals in many-valued logic.  By 
   \cite[Theorem~2.2]{web}, and
  Proposition 
\ref{proposition:trivial}(iv),  no interval algebra
arising from the constructions of \cite{web}
can be an  IMV-algebra.
 The paper \cite{chakue} deals with  ``interval MV-algebras'',
having no relations to IMV-algebras. 
The main result of  \cite{chakue}  is
  that every interval  $[\alpha,\beta]$ in an MV-algebra
$A$ can be equipped with the structure of an MV-algebra.
Again by  Proposition~\ref{proposition:trivial}(iv), no ``interval MV-algebra'' in the
sense of \cite{chakue}  can be the subreduct of an  IMV-algebra.
\subsection*{Interval t-norms on   lattices, orders, logics}
The book chapter  in \cite{35years} 
entitled
``Interval-Valued Algebras and Fuzzy Logics''
investigates  ``triangularization'' as an operator
from
 bounded lattices into  bounded lattices. 
 Also see  \cite{van-bis}.   The action   of 
 triangularization on bounded lattices homomorphisms
  is not considered. 
In the same paper,   triangularization 
is applied to classes of residuated lattices.
Triangularization operators
of general lattice-ordered  structures are also considered
in \cite{bedbed}, using a different terminology.
  ``Interval {\it functors}''  from t-normed bounded
  lattices into t-norm bounded lattices
   are  explicitly considered, e.g., 
  in  
\cite[Subsection 6.1]{bedbed}. In
  \cite{natal-LNCS} the authors
define a category whose objects are t-norms  on arbitrary
bounded lattices and whose  morphisms are suitable
generalizations of automorphisms.
We refer to
\cite{natal-LNCS,
information,
 cina,  [21]},  
 and to  the bibliography of
 \cite{35years}    for an account of
interval  t-norms on classes  of
bounded lattices. 
Similarly as in IMV-algebras, the
  product and the inclusion order
are naturally definable in most  interval
algebras arising from triangularizations,
   \cite{gehwal,  cina, web, information}. 
  One can  find in the literature  
  various interval-valued logics,   \cite{information, 35years, van-bis}.
For instance, the semantics of the
logic considered in  
\cite[Section 4.2]{35years} 
   is in terms of  ``triangle algebras'', and 
   satisfies a soundness and
   completeness theorem.  For any of these logic systems one
   can naturally ask about the algorithmic  complexity of
   the tautology/consequence problem.

\bibliographystyle{plain}

\end{document}